\documentclass[submission,final,copyright,creativecommons]{eptcs}
\AddToHook{package/refstyle/before}{\DeclareRobustCommand\eqref{\relax}\let\eqref\relax}

\usepackage{paralist}

\def\sagdir{sag}
\def\rootdir{.}
\usepackage{import}

\usepackage[english]{babel} %
\usepackage[style=russian]{csquotes}

\usepackage[usenames,dvipsnames,table]{xcolor}
\usepackage[outline]{contour}

\usepackage{ifthen}
\usepackage{iftex}
\usepackage{ifxetex}
\ifxetex
\else
    \usepackage[utf8]{inputenc}
\fi

\usepackage{subcaption}

\usepackage{blkarray, array}

\usepackage[T1]{fontenc}
\usepackage{lipsum}
\usepackage{xspace}
\usepackage{amsmath}
\usepackage{amsthm}
\usepackage{amssymb} %
\usepackage{mathtools}

\usepackage[capitalize]{cleveref}

\usepackage{multirow}
\usepackage{prftree}

\usepackage{comment}
\usepackage{accents}

\usepackage{ebproof}
\usepackage{units}

\usepackage{todonotes}

\usepackage{pgfplots}
\pgfplotsset{compat=1.16}
\usepackage[framemethod=TikZ]{mdframed}
\usepackage{graphicx}
\usepackage{calc}

\usepackage{adjustbox}
\usepackage[section]{placeins}
\usepackage{eso-pic}
\usepackage{capt-of}
\usepackage{wrapfig}
\usepackage[normalem]{ulem}
\usepackage{newclude}
\usepackage{refstyle}

\usepackage{stmaryrd}
\usepackage{pifont}
\usepackage{relsize}

\usepackage{makeidx}

\usepackage{longtable}

\usepackage{multicol}

\usepackage{caption}
\usepackage{transparent}
\usepackage{datetime2}
\usepackage{anyfontsize}
\usepackage{ifmtarg}

\usepackage{underscore}

\graphicspath{
    {\rootdir/figures/},
    {\rootdir/pics/},
    {\rootdir/slides/},
    {\rootdir/videos/generated/},
    {\rootdir/figures/reits/},
    {\rootdir/figures/firstpaper/},
    {\rootdir/figures/uncertainty/},
    {\rootdir/papers/uncertainty/},
    {},
    {../}}

\newboolean{statuscolors}
\newboolean{instructors}
\newboolean{showslides} %
\newboolean{showproofs}
\newboolean{debugimages}
\newboolean{cachepdf}
\newboolean{devel} %

\newboolean{codeexercises}
\newboolean{githublink}

\newcommand{\instructors}[1]{\ifbool{instructors}{%
        #1
    }{}}

\newcommand{\showslides}[1]{\ifbool{showslides}{#1}{}}

\newcommand{\showproofs}[1]{\ifbool{showproofs}{#1}{}}
\newcommand{\devel}[1]{\ifbool{devel}{%
        \clearpage
        \pagecolor{yellow!7}
        #1
        \clearpage
        \pagecolor{white}
    }{}
}
\newcommand{\toexclude}[1]{\ifbool{devel}{%
        \clearpage
        \pagecolor{gray}
        #1
        \clearpage
        \pagecolor{white}
    }{}
}

\newcommand{\codeexercises}[1]{%
    \ifbool{codeexercises}{%
        \FloatBarrier
        \clearpage
        \pagecolor{blue!7}
        #1
        \clearpage
        \pagecolor{white}
    }{}%
}

\newlength{\dpiconwidth}
\crefformat{equation}{(#2#1#3)}
\crefformat{equation}{Equation~(#2#1#3)}

\usepackage[outer]{showlabels}

\newcommand{\linkgithub}{
    \ifbool{githublink}{%
        \usebox{\urlbox}%
    }{}
}
\newcommand{\draftnotice}{
    \ifbool{githublink}{%
        \footnotesize You are reading a draft compiled on \DTMnow%
    }{}%
}

\let\oldsubseteq\subseteq
\renewcommand{\subseteq}{\mathrel{\oldsubseteq}\linebreak[0]}
\let\oldsupseteq\supseteq
\renewcommand{\supseteq}{\mathrel{\oldsupseteq}\linebreak[0]}
\let\oldsubset\subset
\renewcommand{\subset}{\mathrel{\oldsubset}\linebreak[0]}
\let\oldsupset\supset
\renewcommand{\supset}{\mathrel{\oldsupset}\linebreak[0]}

\usepackage{filemod}

\usepackage{trimclip}

\usepackage{ulem}
\usepackage{silence}
\usepackage{tikz}

\usetikzlibrary{positioning,
  intersections,
  hobby,
  patterns,
  calc,
  decorations.pathmorphing,
  decorations.markings,
  shadows,
  shapes,
  cd,
  decorations.markings,
  positioning,
  arrows.meta,
  shapes,
  calc,
  fit,
  quotes,
backgrounds}

\usepackage{tikzpagenodes}

\tikzcdset{arrow style=tikz}
\tikzstyle{arrowmorphismstyle} = [-{Triangle[width=3.5pt,length=3.5pt]},draw=morphisms]
\tikzstyle{arrownorphismstyle} = [-{Triangle[width=3.5pt,length=3.5pt]},dashed, draw=norphisms]
\tikzstyle{arrowmorphismstylemapsto} = [|-{Triangle[width=3.5pt,length=3.5pt]},draw=morphisms]
\tikzstyle{arrowfunctorstyle} = [-{Triangle[width=3.5pt,length=3.5pt]},draw=functors]
\tikzstyle{block} = [draw=morphisms, rectangle, minimum height=2em, minimum width=3em, thick]
\tikzstyle{blockgroup} = [draw=morphisms, rectangle, minimum height=2em, minimum width=3em, dotted, thick]
\tikzstyle{resblock} = [draw=red, rectangle, minimum height=2em, minimum width=3em, thick]
\tikzstyle{compblock} = [draw=blue, rectangle, minimum height=2em, minimum width=3em, thick]
\tikzstyle{effblock} = [draw=black, rectangle, minimum height=2em, minimum width=3em, thick]
\tikzstyle{tupblock} = [draw=morphisms, rectangle, minimum height=2em, minimum width=3em, thick]
\tikzstyle{styleobjects} = [draw=objects, thick]
\tikzstyle{morblock} = [draw=morphisms, rectangle, minimum height=2em, minimum width=3em, thick, rounded corners=2]
\tikzstyle{norblock} = [draw=norphisms, rectangle, minimum height=2em, minimum width=3em, thick, rounded corners=2]

\tikzstyle{blockk} = [draw, rectangle, minimum height=2.5em, minimum width=3.5em]
\tikzstyle{block1} = [draw, rectangle, minimum height=1.5em, minimum width=2.5em]
\tikzstyle{blockDyn} = [draw, rectangle, minimum height=2.5em, minimum width=3.5em, align=center, inner sep=10pt, thick, fill=white, copy shadow={draw=black,fill=black,opacity=1,shadow xshift=0.5ex,shadow yshift=-0.5ex}]
\tikzstyle{blockAlg} = [draw, rectangle, minimum height=1.5em, minimum width=2.5em, align=center, inner sep=10pt, thick]
\tikzstyle{sum} = [draw,circle]
\tikzstyle{input} = [coordinate]
\tikzstyle{output} = [coordinate]
\tikzstyle{pinstyle} = [pin edge={to-,thin,black}]
\tikzcdset{every label/.append style={font=\normalsize}}

\tikzset{fcname/.store in =\fcname, fcname={}}
\tikzset{funame/.store in =\funame, funame={}}
\tikzset{rcname/.store in =\rcname, rcname={}}
\tikzset{runame/.store in =\runame, runame={}}
\tikzset{whereres/.store in =\whereres, whereres=0.5}
\tikzset{wherefun/.store in =\wherefun, wherefun=0.5}
\tikzset{relres/.store in =\relres, relres={right}}
\tikzset{relfun/.store in =\relfun, relfun={left}}
\tikzset{posres/.store in =\posres, posres=0.95}
\tikzset{posfun/.store in =\posfun, posfun=0.95}
\tikzset{loos/.store in =\loos, loos=2}
\tikzset{feedback/.store in =\feedback, feedback=0}
\tikzset{
  DP/.style={%
    label/.style={
      font=\everymath\expandafter{\the\everymath\scriptstyle},
      inner sep=5pt,
      node distance=2pt and -2pt},
    semithick,
    node distance=1 and 1,
    rconn/.style={color=white,opacity=0.0,postaction={decorate}, shorten <=3.2pt, shorten >= 0.8,
    decoration={markings,
    mark= at position 0 with {
      \coordinate (a);
    },
    mark=at position .5 with
        {
        \ifthenelse{\equal{\feedback}{1}}{\def\angleOut{-90}\def\angleIn{-90}}{\def\angleOut{0}\def\angleIn{180}}
        \coordinate (b);
        \draw[dashed,\dpred,opacity=1.0] (a) to[out=\angleOut,in=\angleIn,looseness=\loos]
        node[pos=\posres,\relres=\whereres mm,\dpred,opacity=1,fill=white,inner sep=1pt,outer sep=1pt]{\rcname} (b);
      },
      mark= at position 1 with
        {
        \ifthenelse{\equal{\feedback}{1}}{\def\angleOut{0}\def\angleIn{0}}{\def\angleOut{180}\def\angleIn{0}}
        \ifthenelse{\equal{\feedback}{1}}{\def\symbol{\succeq}}{\def\symbol{\preceq}}
        \coordinate (c);
        \draw[\dpgreen,opacity=1.0] (c) to[out=\angleOut,in=\angleIn,looseness=\loos]
        node[pos=\posfun,\relfun=\wherefun mm,\dpgreen,opacity=1,fill=white,inner sep=1pt,outer sep=1pt]{\fcname} (b){}; %
        \node[draw,circle,inner sep=0.5pt,color=black,fill=white,opacity=1.0,line width=0.2mm] at (b) (nodepreceq) {$\symbol$};
      }
    }},
    runconn/.style={color=\dpred,dashed,postaction={decorate},
    decoration={markings,
    mark= at position 1 with {
      \coordinate (a);
      \draw[\dpred,opacity=1.0,dashed] ($(a)+(0.05,0)$) --++ (0.5,0) node[\relres,pos=\posres]{\runame};}
    }
    },
    funconn/.style={color=white,postaction={decorate},
    decoration={markings,
    mark= at position 0 with {
      \coordinate (a);
      \draw[\dpgreen] ($(a)+(-0.05,0)$) -- ($(a)+(-0.5,0)$) node[\relfun,pos=\posfun]{\funame};}
    }
    },
    execute at begin picture={\tikzset{
      x=\dpx, y=\dpy,
      every fit/.style={inner xsep=\dpx, inner ysep=\dpy}}}
  },
  dpx/.store in=\dpx,
  dpx = 1.5cm,
  dpy/.store in=\dpy,
  dpy = 1.5ex,
  dp port sep/.store in=\dpportsep,
  dp port sep=2,
  dp port length/.store in=\dpportlen,
  dp port length=4pt,
  dp min width/.store in=\dpminwidth,
  dp min width=0.5cm,
  dp rounded corners/.store in=\dpcorners,
  dp rounded corners=2pt,
  dp small/.style={dp port sep=1, dp port length=2.5pt, dpx=.4cm, dp min width=.4cm, dpy=.7ex},
  dp/.code 2 args={%
    \pgfmathsetlengthmacro{\dpheight}{\dpportsep * (max(#1,#2)) * \dpy}
    \pgfkeysalso{draw,%
      minimum width=\dpminwidth,%
      minimum height=\dpheight,%
      outer sep=0pt,%
      inner sep=5pt,%
      rounded corners=\dpcorners,
      thick,
      prefix after command={\pgfextra{\let\fixname\tikzlastnode}},
      append after command={\pgfextra{\draw
      \ifnum #1=0{} \else foreach \i in {1,...,#1} {
        ($(\fixname.north west)!{\i/(#1+1)}!(\fixname.south west)$) +(0,0) node[solid,left,circle,color=\dpgreen,draw,fill=\dpgreen,scale=0.3] {} coordinate (\fixname_fun\i) -- +(0,0) coordinate (\fixname_fun\i')}\fi %
      \ifnum #2=0{} \else foreach \i in {1,...,#2} {
        ($(\fixname.north east)!{\i/(#2+1)}!(\fixname.south east)$) +(0,0) coordinate (\fixname_res\i') -- +(0,0) node[solid,right,circle,color=\dpred,draw,fill=\dpred,scale=0.3] {} coordinate (\fixname_res\i)}\fi;
      }}}
  },
  dp name/.style={append after command={\pgfextra{\node[label=center,inner sep=2pt,fill=white] at (\fixname) {#1};}}}
}

\tikzset{
  oriented WD/.style={%
    every to/.style={out=0,in=180,draw},
    label/.style={
      font=\everymath\expandafter{\the\everymath\scriptstyle},
      inner sep=0pt,
      node distance=2pt and -2pt},
    semithick,
    node distance=1 and 1,
    decoration={markings, mark=at position .5 with {\arrow{stealth};}},
    ar/.style={postaction={decorate}},
    execute at begin picture={\tikzset{
      x=\bbx, y=\bby,
      every fit/.style={inner xsep=\bbx, inner ysep=\bby}}}
  },
  bbx/.store in=\bbx,
  bbx = 1.5cm,
  bby/.store in=\bby,
  bby = 1.75ex,
  bb port sep/.store in=\bbportsep,
  bb port sep=2,
  bb port length/.store in=\bbportlen,
  bb port length=4pt,
  bb min width/.store in=\bbminwidth,
  bb min width=1cm,
  bb rounded corners/.store in=\bbcorners,
  bb rounded corners=2pt,
  bb small/.style={bb port sep=1, bb port length=2.5pt, bbx=.4cm, bb min width=.4cm, bby=.7ex},
  bb/.code 2 args={%
    \pgfmathsetlengthmacro{\bbheight}{\bbportsep * (max(#1,#2)+1) * \bby}
    \pgfkeysalso{draw,minimum height=\bbheight,minimum width=\bbminwidth,outer sep=0pt,
      rounded corners=\bbcorners,thick,
      prefix after command={\pgfextra{\let\fixname\tikzlastnode}},
      append after command={\pgfextra{\draw
      \ifnum #1=0{} \else foreach \i in {1,...,#1} {
        ($(\fixname.north west)!{\i/(#1+1)}!(\fixname.south west)$) +(-\bbportlen,0) coordinate (\fixname_in\i) -- +(\bbportlen,0) coordinate (\fixname_in\i')}\fi %
      \ifnum #2=0{} \else foreach \i in {1,...,#2} {
        ($(\fixname.north east)!{\i/(#2+1)}!(\fixname.south east)$) +(-\bbportlen,0) coordinate (\fixname_out\i') -- +(\bbportlen,0) coordinate (\fixname_out\i)}\fi;
      }}}
  },
  bb name/.style={append after command={\pgfextra{\node[anchor=north] at (\fixname.north) {#1};}}}
}

\tikzset{
  tick/.style={postaction={
    decorate,
    decoration={markings, mark=at position 0.5 with {\draw[-] (0,.4ex) -- (0,-.4ex);}}}
  }
}

\newcommand{\includesag}[1]{%
  \saginclude{#1}
}

\newcommand{\saginclude}[1]{%
    \def\pdffile{\sagdir/\detokenize{#1}.pdf}%
    \def\srcfile{\sagdir/\detokenize{#1}.tikz}%
    \ifbool{cachepdf}{%
        \IfFileExists{\pdffile}{%
            \Filemodcmp{\pdffile}{\srcfile}{%
                \includegraphics[scale=1]{\pdffile}%
            }{%
                \PackageWarning{ACT4E}{TikZ file is more recent than outdated pdf cache \pdffile}%
                \input{\srcfile}%
            }%
        }{\input{\srcfile}}%
    }{%
        \input{\srcfile}%
    }%
}

\usetikzlibrary{calc}

\newlength{\brickwidth}
\newlength{\brickheight}
\newlength{\brickdia}
\newlength{\brickdiaheight}
\newlength{\brickmultipliedx}
\newlength{\brickmultipliedy}
\newlength{\halfbrickwidth}
\newlength{\hatchspread}
\newlength{\hatchthickness}
\newlength{\hatchshift}
\newcommand{\hatchcolor}{}
\tikzset{hatchspread/.code={\setlength{\hatchspread}{#1}},
  hatchthickness/.code={\setlength{\hatchthickness}{#1}},
  hatchshift/.code={\setlength{\hatchshift}{#1}},%
  hatchcolor/.code={\renewcommand{\hatchcolor}{#1}}}
\tikzset{hatchspread=3pt,
  hatchthickness=0.4pt,
  hatchshift=0pt,%
  hatchcolor=black}
\pgfdeclarepatternformonly[\hatchspread,\hatchthickness,\hatchshift,\hatchcolor]%
{custom north west lines}%
{\pgfqpoint{\dimexpr-2\hatchthickness}{\dimexpr-2\hatchthickness}}%
{\pgfqpoint{\dimexpr\hatchspread+2\hatchthickness}{\dimexpr\hatchspread+2\hatchthickness}}%
{\pgfqpoint{\dimexpr\hatchspread}{\dimexpr\hatchspread}}%
{%
  \pgfsetlinewidth{\hatchthickness}
  \pgfpathmoveto{\pgfqpoint{0pt}{\dimexpr\hatchspread+\hatchshift}}
  \pgfpathlineto{\pgfqpoint{\dimexpr\hatchspread+0.15pt+\hatchshift}{-0.15pt}}
  \ifdim \hatchshift > 0pt
  \pgfpathmoveto{\pgfqpoint{0pt}{\hatchshift}}
  \pgfpathlineto{\pgfqpoint{\dimexpr0.15pt+\hatchshift}{-0.15pt}}
  \fi
  \pgfsetstrokecolor{\hatchcolor}
  \pgfusepath{stroke}
}

\pgfdeclarelayer{background}
\pgfdeclarelayer{foreground}
\pgfsetlayers{background,main,foreground}

\newcommand{\rope}[7]{
\begin{tikzpicture}[n/.style={circle,draw,minimum size=2mm,inner sep=0, fill=#1}, remember picture]
\ifthenelse{\lengthtest{#5 pt=0 pt} \AND \lengthtest{#6 pt=1 pt}}
{\draw[ultra thick, #1] (0,0) -- node[pos=#5, #1,n]{} node[pos=#6,n] {} node[pos=0.5, above, black]{#7} (#3,#4){};}
{\draw[ultra thick, #1] (#5*#3,#5*#4) -- node[pos=0, #1,n,overlay]{} node[pos=1, #1,n] {} node[pos=0.5, above, black]{#7} (#6*#3,#6*#4){};
\begin{pgfonlayer}{background}
\draw[ultra thick, #2] (0,0)--(#5*#3,#5*#4){};
\draw[ultra thick, #2] (#6*#3,#6*#4)--(#3,#4){};
\end{pgfonlayer}}
\end{tikzpicture}}

\newcommand{\northenasymb}{%
\mathbin{\begin{tikzpicture}[n/.style={circle,draw,minimum size=0.8mm,inner sep=0, fill=black}]
\node[n] at (0.125,0.05) {};
\node at (0.1,0.1){};
\draw[thick](0,0.05)--(0.25,0.05){};
\end{tikzpicture}}}

\newcommand{\northenbsymb}{%
\mathbin{\begin{tikzpicture}[n/.style={circle,draw,minimum size=0.8mm,inner sep=0, fill=black}]
\node[n] at (0,0.05) {};
\node[n] at (0.2,0.05) {};
\node at (0.1,0.1){};
\draw[thick](0,0.05)--(0.2,0.05){};
\end{tikzpicture}}}

\newcommand{\northena}[2]{{#1 \northenasymb #2}}
\newcommand{\northenb}[2]{{#1 \northenbsymb  #2}}

\usepackage{forest}

\newcommand{\dirtree}[1]{%
    \begin{forest}
      for tree={
        font=\ttfamily,
        grow'=0,
        child anchor=west,
        parent anchor=south,
        anchor=west,
        calign=first,
        edge path={
          \noexpand\path [draw, \forestoption{edge}]
          (!u.south west) +(4.5pt,0) |- node[fill,inner sep=1.25pt] {} (.child anchor)\forestoption{edge label};
        },
        before typesetting nodes={
          if n=1
            {insert before={[,phantom]}}
            {}
        },
        fit=band,
        before computing xy={l=15pt},
      }
    #1
    \end{forest}
}

\definecolor{cblue}{rgb}{0.3,0.3,1.0}
\definecolor{cred}{rgb}{1.0,0.1,0.1}

\definecolor{cwhite}{rgb}{0.95,0.95,0.97}
\definecolor{funbox}{rgb}{0.92,1,0.9}
\definecolor{resbox}{rgb}{1,0.74,0.75}
\definecolor{impbox}{rgb}{1,0.92,0.85}

\usetikzlibrary{shadows.blur}
\usetikzlibrary{shapes.symbols}

\definecolor{red1}{rgb}{0.30,0,0}
\definecolor{red2}{rgb}{0.60,0,0}
\definecolor{red3}{rgb}{0.90,0,0}
\definecolor{my-cyan}{cmyk}{1,0,0,0}
\definecolor{my-magenta}{cmyk}{0,1,0,0}
\definecolor{my-yellow}{cmyk}{0,0,1,0}

\tikzcdset{
    setcd/.style={
        every matrix/.append style={
            draw=setcolorbord,
            fill=setcolor,
            rounded corners,
            inner sep=1pt,
            #1
        },
    },
}

\tikzcdset{
    posetcd/.style={
        every matrix/.append style={
            draw=posetcolorbord,
            fill=posetcolor,
            rounded corners,
            inner sep=1pt,
            #1
        },
    },
}

\tikzcdset{
    catcd/.style={
        every matrix/.append style={
            draw=catcolor,
            fill=catcolor,
            rounded corners,
            inner sep=4pt,
            #1
        },
    },
}

\tikzcdset{
    graphcd/.style={
        every matrix/.append style={
            draw=graphcolorbord,
            fill=graphcolor,
            rounded corners,
            inner sep=4pt,
            #1
        },
    },
}

\newcommand{\stylesets}[1]{{\ColorIfNowBlack{formulasetcolor}{\mathbf{#1}}}}
\newcommand{\styleelements}[1]{{\color{elementscolor}{#1}}}
\newcommand{\stylemaps}[1]{{\stylemorph{#1}}}
\newcommand{\EM}[1]{\ensuremath{#1}\xspace}
\newcommand{\setAel}{\styleelements{a}}

\newcommand{\manifold}[1]{\mathcal{M}#1}

\newcommand{\setBel}{\styleelements{b}}
\newcommand{\mapelevation}{{h}}

\newcommand{\setunion}{\mathbin{{\color{formulasetcolor}\cup}}} %
\newcommand{\setinsymbol}{\in}
\newcommand{\setin}{\linebreak[0]\mathrel{{\color{formulasetcolor}\setinsymbol}}\linebreak[0]} %

\newcommand{\setsubseteq}{\mathrel{\color{formulasetcolor}{\subseteq}}} %

\newcommand{\setsubset}{\mathrel{\color{formulasetcolor}{\subset}}} %
\newcommand{\realsname}{R}
\newcommand{\reals}{\EM{\stylesets{\mathbb{\realsname}}}} %
\newcommand{\natnumbersname}{N}
\newcommand{\natnumbers}{\EM{\stylesets{\mathbb{\natnumbersname}}}} %

\newcommand{\nonNegReals}{\EM{\reals_{\geq0}}} %
\newcommand{\singletonsymbol}{{1}}
\newcommand{\singleton}{\stylesets{\singletonsymbol}} %
\newcommand{\singletonel}{\bullet} %
\newcommand{\powersetmapname}{Pow}
\DeclareMathOperator{\powerset}{\powersetmap}
\newcommand{\powersetmap}{\stylesets{\aword{\powersetmapname}}}
\newcommand{\cartprod}{\mathbin{\stylesets{\raisedtexttimes}}} %

\newcommand{\funcprod}{\mathbin{\mathcolor{morphisms}{\times}}} %

\newcommand{\stylepos}[1]{{\mathcolor{colorposetsymbol}{#1}}}
\newcommand{\posAlet}{P}
\newcommand{\posAeln}{p}
\newcommand{\posBlet}{Q}
\newcommand{\posBeln}{q}
\newcommand{\posClet}{R}
\newcommand{\posCeln}{r}

\newcommand{\makesymbolpos}[1]{\EM{\stylepos{\mathbf{#1}}}}

\newcommand{\posA}{\makesymbolpos{\posAlet}}

\newcommand{\posgenA}{\mathbf{\posAlet}}
\newcommand{\posAel}{\styleelements{\posAeln}}
\newcommand{\posAnel}[1]{\styleelements{\posAeln_{#1}}}
\newcommand{\posB}{\makesymbolpos{\posBlet}}
\newcommand{\posgenB}{\mathbf{\posBlet}}
\newcommand{\posBel}{\styleelements{\posBeln}}

\newcommand{\posC}{\makesymbolpos{\posClet}}
\newcommand{\posgenC}{\mathbf{\posClet}}
\newcommand{\posCel}{\styleelements{\posCeln}}
\newcommand{\funposA}{\F{\posgenA}}
\newcommand{\funposB}{\F{\posgenB}}

\newcommand{\funposAel}{\F{\posAeln}}
\newcommand{\funposAopel}{\F{\posAeln\opel}}
\newcommand{\funposBopel}{\F{\posBeln\opel}}

\newcommand{\resposA}{\R{\posgenA}}
\newcommand{\resposB}{\R{\posgenB}}
\newcommand{\resposC}{\R{\posgenC}}

\newcommand{\resposAel}{\R{\posAeln}}
\newcommand{\resposBel}{\R{\posBeln}}
\newcommand{\resposCel}{\R{\posCeln}}
\newcommand{\interv}[2]{{\colL[}#1,\smallspace#2{\colU]}}
\definecolor{functionality}{rgb}{0.094869,0.500000,0.000000}
\definecolor{requirements}{rgb}{0.555789,0.000000,0.000000}
\definecolor{implementations}{RGB}{214,120,5}
\definecolor{functionalitylight}{rgb}{0.094869,0.800000,0.000000}
\definecolor{requirementslight}{rgb}{0.8,0.000000,0.000000}
\definecolor{implementationslight}{RGB}{214,120,5}
\definecolor{listcolor}{named}{black}
\definecolor{setcolor}{RGB}{243, 180, 126} %
\definecolor{setcolorbord}{RGB}{253, 143, 47}
\definecolor{posetcolor}{rgb}{0.8, 0.925, 0.992} %

\definecolor{posetcolorbord}{RGB}{81, 171, 252} %
\definecolor{graphcolor}{RGB}{252, 246, 207}%
\definecolor{graphcolorbord}{RGB}{219, 213, 173} %
\definecolor{formulasetcolor}{named}{brown} %
\definecolor{formulasetLcolor}{RGB}{213,87,96} %
\definecolor{elementscolor}{rgb}{0.1,0,0.4}
\definecolor{catcolor}{rgb}{0.94,1.0,0.94}
\definecolor{darkgreen}{rgb}{0.0, 0.5, 0.0}
\definecolor{darkred}{rgb}{0.5, 0.0, 0.0}
\definecolor{brightpink}{rgb}{1.0, 0.65, 0.79}
\definecolor{applegreen}{rgb}{0.55, 0.71, 0.0}
\definecolor{custompurple}{rgb}{0.18,0.46,0.74}
\definecolor{custompink}{rgb}{0.73,0.83,0.91}
\definecolor{funcname}{named}{blue}
\definecolor{classname}{named}{blue}
\definecolor{fieldname}{named}{darkgreen}
\definecolor{parula1}{RGB}{53,42,135}
\definecolor{parula2}{RGB}{15,92,221}
\definecolor{parula3}{RGB}{18,125,216}
\definecolor{parula4}{RGB}{7,156,207}
\definecolor{parula5}{RGB}{21,177,180}
\definecolor{parula6}{RGB}{21,177,180}
\definecolor{parula7}{RGB}{89,189,140}
\definecolor{parula8}{RGB}{165,190,107}
\definecolor{parula9}{RGB}{225,185,82}

\definecolor{dpcolor}{rgb}{0.5,0,0.5}
\def\dpgreen{functionality}
\def\dpred{requirements}

\definecolor{darkblue}{rgb}{0.0, 0.0, 0.55}
\definecolor{shadecolor}{rgb}{0.68, 0.85, 0.9}
\definecolor{transmuter}{RGB}{46,117,189}
\definecolor{objects}{RGB}{213,87,96}
\definecolor{morphisms}{RGB}{46,117,189}
\definecolor{norphisms}{RGB}{213,87,96}
\definecolor{colorposetsymbol}{RGB}{90,90,90}

\definecolor{transmuted}{RGB}{213,87,96}
\definecolor{naturaltransformations}{rgb}{0.1,0.5,0.0}
\definecolor{functors}{rgb}{0.5,0,0.5}

\definecolor{instructors}{RGB}{214,120,5}
\definecolor{upcolor}{RGB}{252, 3, 107}
\definecolor{downcolor}{named}{Orange}
\definecolor{antichaincolor}{named}{Brown}
\definecolor{staincola}{RGB}{181,23,0}
\definecolor{staincolb}{RGB}{255,150,141}
\definecolor{staincolc}{RGB}{181,23,30}
\definecolor{staincold}{RGB}{241,219,200}
\definecolor{staincole}{RGB}{92,24,13}
\definecolor{staincolf}{RGB}{29,177,0}
\definecolor{ropecola}{RGB}{187,127,14}
\definecolor{ropecolb}{RGB}{97,216,54}
\newcommand{\colTransmuted}{\color{transmuted}}
\newcommand{\F}[1]{{\mathcolor{\dpgreen}{#1}}}
\newcommand{\R}[1]{{\mathcolor{\dpred}{#1}}}

\newcommand{\styleobj}[1]{{\colTransmuted{#1}}}
\newcommand{\stylemorph}[1]{{\mathcolor{morphisms}{#1}}}
\newcommand{\stylenorph}[1]{{\mathcolor{norphisms}{#1}}}
\newcommand{\stylefunctors}[1]{{\mathcolor{functors}{#1}}}

\newcommand{\stylendp}[1]{{\mathcolor{norphisms}{{#1}}}}

\newcommand{\mthensymbol}{\fatsemi}
\newcommand{\mthen}{\mathbin{\stylemorph{\mthensymbol}}\linebreak[0]} %
\newcommand{\mthenof}[1]{\mathbin{\stylemorph{\mthensymbol}_{#1}}\linebreak[0]} %
\newcommand{\fthensymbol}{\fatsemi}
\newcommand{\fthen}{\mathbin{\stylefunctors{\fthensymbol}}} %

\newcommand{\postcomp}[1]{{\mathrm{post}}_#1}
\newcommand{\precomp}[1]{{\mathrm{pre}}_#1}

\newcommand{\Obname}{Ob}
\newcommand{\Ob}{\styleobj{\mathrm{\Obname}}}
\newcommand{\catidname}{\idname}
\newcommand{\catid}{\stylemorph{\catidname}} %

\newcommand{\catidof}[1]{{\catid}_{#1}} %
\newcommand{\catidat}[1]{{\catid}_{#1}} %

\newcommand{\CatSymbol}[1]{\textbf{#1}\xspace}
\newcommand{\Homname}{\textrm{Hom}}
\newcommand{\Hom}{\stylefunctors{\Homname}}

\newcommand{\HomSet}[3]{\Hom_{#1}\pars{{#2};\linebreak[0]{#3}}}
\newcommand{\CatB}{\CatSymbol{B}}
\newcommand{\CatC}{\CatSymbol{C}}
\newcommand{\CatE}{\CatSymbol{E}}
\newcommand{\CatV}{\CatSymbol{V}}

\newcommand{\Obof}[1]{\Ob_{#1}}
\newcommand{\ObC}{\Obof\CatC}
\newcommand{\ObE}{\Obof\CatE}

\newcommand{\Obja}{\styleobj{X}}

\newcommand{\Objb}{\styleobj{Y}}

\newcommand{\Objc}{\styleobj{Z}}
\newcommand{\Objd}{\styleobj{U}}

\newcommand{\nora}{\stylenorph{n}}
\newcommand{\norb}{\stylenorph{m}}
\newcommand{\norc}{\stylenorph{o}}

\newcommand{\mora}{\stylemorph{f}}
\newcommand{\morb}{\stylemorph{g}}
\newcommand{\morc}{\stylemorph{h}}

\newcommand{\funa}{\stylefunctors{F}}
\newcommand{\funb}{\stylefunctors{G}}

\newcommand{\idname}{\textrm{id}}
\newcommand{\nidname}{\textrm{nid}}

\newcommand{\funidsymbol}{\idname}
\newcommand{\funid}{\stylefunctors{\funidsymbol}} %
\newcommand{\funidat}[1]{\funid_{#1}}

\newcommand{\mapidname}{\idname}
\newcommand{\mapid}{\stylemorph{\mapidname}} %

\newcommand{\posopsymbol}{op}
\newcommand{\posop}{^{\stylepos{\textrm{\posopsymbol}}}}

\newcommand{\op}{^{\textrm{op}}}
\newcommand{\opel}{^*}

\newcommand{\mtimescatob}{\mathbin{\styleobj{\pmb{\otimes}}}} %
\newcommand{\mtimescat}{\mathbin{\stylefunctors{\pmb{\otimes}}}} %
\newcommand{\idmoncatsymbol}{1}
\newcommand{\idmoncat}{\styleobj{\mathbf{\idmoncatsymbol}}} %

\newcommand{\leftunitor}{\stylenat{\aword{lu}}} %
\newcommand{\rightunitor}{\stylenat{\aword{ru}}} %
\newcommand{\associator}{\stylenat{\aword{as}}} %

\newcommand{\true}{\top} %
\newcommand{\false}{\bot} %
\newcommand{\boolandsymbol}{\wedge}
\newcommand{\boolorsymbol}{\vee}
\newcommand{\booland}{\mathbin{\boolandsymbol}} %
\newcommand{\boolor}{\mathbin{\boolorsymbol}} %

\newcommand{\Imp}{\mathrel{\Rightarrow}} %
\newcommand{\mto}{\mathrel{{\stylemorph{\to}}}\linebreak[0]} %
\newcommand{\mtoin}[1]{\mathrel{{\stylemorph{\to}}_{#1}}\linebreak[0]}

\newcommand{\norto}{\mathrel{\stylenorph{\dashrightarrow}}}
\newcommand{\mfromin}[1]{\mathrel{{\stylemorph{\leftarrow}}_{#1}}\linebreak[0]}

\newcommand{\slashedrightarrow}{\relbar\joinrel\mapsto}
\newcommand{\profto}{\mathrel{\slashedrightarrow}\linebreak[0]} %

\newcommand{\nprofto}{\mathcolor{red}{\mathrel{\slashedrightarrow}}\linebreak[0]}

\newcommand{\toinPos}{\mtoin\Pos}

\newcommand{\toinRel}{\mtoin\Rel}

\newcommand{\funsp}{\EM{{\colF \mathbf{F}}}} %

\newcommand{\ressp}{\EM{{\colR\mathbf{R}}}} %

\newcommand{\posleqcore}{\stylemorph{\preceq}}
\newcommand{\pospreccore}{\stylemorph{\prec}}
\newcommand{\possucccore}{\stylemorph{\succ}}

\newcommand{\posleqof}[1]{\mathrel{{\posleqcore}_{#1}}}
\newcommand{\posprecof}[1]{\mathrel{{\pospreccore}_{#1}}}
\newcommand{\possuccof}[1]{\mathrel{{\possucccore}_{#1}}}

\def\bitcoinA{%
    \leavevmode
    \vtop{\offinterlineskip %
    \setbox0=\hbox{B}%
    \setbox2=\hbox to\wd0{\hfil\hskip-.03em
    \vrule height .3ex width .15ex\hskip .08em
    \vrule height .3ex width .15ex\hfil}
        \vbox{\copy2\box0}\box2}}

\newcommand{\adp}{\adpa} %
\newcommand{\adpa}{\styledp{\mathbf{d}}} %
\newcommand{\adpb}{\styledp{\mathbf{e}}} %

\newcommand{\andpa}{\stylendp{{n}}}

\newcommand{\DPname}{DP}

\newcommand{\DP}{\CatSymbol{\DPname}} %
\newcommand{\Boolposname}{Bool}

\newcommand{\Bool}{\CatSymbol{\Boolposname}} %
\newcommand{\Categoryname}{Cat}

\newcommand{\Category}{\CatSymbol{\Categoryname}} %
\newcommand{\Relname}{Rel}
\newcommand{\Rel}{\CatSymbol{\Relname}} %

\newcommand{\inrel}[3]{#1\mkern 2mu #2 \mkern 2mu #3} %
\newcommand{\Setname}{Set}
\newcommand{\Set}{\CatSymbol{\Setname}} %
\newcommand{\Posname}{Pos}
\newcommand{\Pos}{\CatSymbol{\Posname}} %

\newcommand{\Berg}{\CatSymbol{Berg}}
\definecolor{alitalia}{named}{darkgreen}
\definecolor{swiss}{named}{darkblue}
\definecolor{ewings}{named}{darkred}

\newcommand{\aword}[1]{\mathsf{#1}}

\newcommand{\definedas}{\mathrel{\coloneqq}} %

\newlength{\ml}
\newlength{\mb}
\newcommand{\lengthfun}{{\operatorname{length}}}
\newcommand{\lengthfunnor}{{\operatorname{length}}}

\newcommand{\tangbundle}[1]{\mathcal{T}#1} %

\newcommand{\defmapperiod}[6]{
    \begin{aligned}
        #1 \colon & #2 & \!\!#3        & \  #4, \\
        & #5 & \!\!\mapsto   & \  #6.
    \end{aligned}
}

\newcommand{\defmapcomma}[6]{
    \begin{aligned}
        #1 \colon & #2 & \!\!#3        & \  #4, \\
        & #5 & \!\!\mapsto   & \  #6,
    \end{aligned}
}

\definecolor{bcolor}{rgb}{0.6,0.8,0.3}

\newcommand{\fatstart}{\llangle}
\newcommand{\fatend}{\rrangle}

\newcommand{\morab}{\mora\mthen\morb}

\newcommand{\steepfun}{\stylemaps{s}}
\newcommand{\steepness}{\sigma}
\newcommand{\steepnessL}{\sigma_\mathrm{L}}
\newcommand{\steepnessU}{\sigma_\mathrm{U}}

\newcommand{\cprod}{\mathbin{\stylemorph{\fatsemi}_{\hspace{-1pt}\stylemorph{\fatstart}}}} %

\newcommand{\tupconcatsymbol}{\fatsemi}
\newcommand{\tupconcat}{\mathbin{\stylemorph{\tupconcatsymbol}_{\hspace{-0.4pt}\stylemorph{\langle}}}} %
\newcommand{\ndpidname}{\nidname}
\newcommand{\ndpid}{\stylendp{\mathbf{\ndpidname}}}

\newcommand{\Nom}{\stylefunctors{\mathrm{Nom}}}

\newcommand{\NomSet}[3]{\Nom_{#1}\pars{{#2};\linebreak[0]{#3}}}

\newcommand{\prcatp}{\CatSymbol{\mathcolor{morphisms}{P}}}

\newcommand{\enid}{\stylemorph{\aword{em}}}
\newcommand{\enidof}[1]{\enid_{#1}}
\newcommand{\enm}{\stylemorph{\aword{ei}}}
\newcommand{\enmof}[3]{\enm_{#1,#2,#3}}

\newcommand{\Ptimes}{\mathbin{\color{colorposetsymbol}\boldtimes}} %

\newcommand{\nomtohom}[2]{{\mathcolor{morphisms}{J}}_{#1#2}}

\crefformat{equation}{(#2#1#3)}
\Crefformat{equation}{Equation~(#2#1#3)}
\crefrangelabelformat{equation}{(#3#1#4) to~(#5#2#6)}
\crefmultiformat{equation}{(#2#1#3)}%
       { and~(#2#1#3)}{, (#2#1#3)}{ and~(#2#1#3)}
\newcommand{\prf}{\expandafter\prftree[r]{\raisebox{3pt}{}}}
\newcommand{\prfdouble}{\expandafter\prftree[double]}
\newcommand{\prfcomma}{\expandafter\prftree[r]{\raisebox{3pt}{,}}}
\newcommand{\prfperiod}{\expandafter\prftree[r]{\raisebox{3pt}{.}}}
\newcommand{\prfsemi}{\expandafter\prftree[r]{\raisebox{3pt}{;}}}
\newcommand{\prfdoublecomma}{\expandafter\prftree[double][r]{\raisebox{3pt}{,}}}
\newcommand{\prfdoubleperiod}{\expandafter\prftree[double][r]{\raisebox{3pt}{.}}}

\makeatletter
\def\mathcolor#1#{\@mathcolor{#1}}
\def\@mathcolor#1#2#3{%
  \protect\leavevmode
  \begingroup
    \color#1{#2}#3%
  \endgroup
}
\makeatother

\newcommand{\gear}[5]{%
\begin{tikzpicture}
\draw[ultra thick]
\foreach \i in {1,...,#1} {%
  [rotate=(\i-1)*360/#1]  (0:#2)  arc (0:#4:#2) {[rounded corners=1.5pt]
             -- (#4+#5:#3)  arc (#4+#5:360/#1-#5:#3)} --  (360/#1:#2)
};\end{tikzpicture}}

\newsavebox{\chaptergear}
\savebox{\chaptergear}{%
  \raisebox{-1pt}{\maxsizebox{2.8mm}{!}{\gear{12}{0.5}{0.4}{10}{2}}}\!\!
}

\DeclareMathSymbol{\shortminus}{\mathbin}{AMSa}{"39}
\newcommand{\minusone}{{{\shortminus}}\kern1pt1}

\makeatletter
\DeclareRobustCommand{\shortto}{%
    \mathrel{\mathpalette\short@to\relax}%
}

\newcommand{\short@to}[2]{%
    \mkern2mu
    \clipbox{{.4\width} 0 0 0}{$\m@th#1\vphantom{+}{\shortrightarrow}$}%
}
\makeatother

\makeatletter

\newcommand{\opensetbracket}{{\ColorIfNowBlack{formulasetcolor}{\mathbf{\{}}}}
\newcommand{\closesetbracket}{{\ColorIfNowBlack{formulasetcolor}{\mathbf{\}}}}}

\newcommand{\cleft}[2][.]{%
    \begingroup\colorlet{savedleftcolor}{.}%
    \color{#1}\mathopen{}\left#2\color{savedleftcolor}%
}
\newcommand{\cright}[2][.]{%
    \color{#1}\right#2\mathclose{}\endgroup%
}

\DeclareRobustCommand{\@makelists}[4]{%
    #1%
    \foreach \entry [count=\xi]  in {#4}{
        \ifnum\xi = 1\else#2\linebreak[0]\fi%
        {{\entry}}%
    }%
    #3%
}
\newcommand{\smallspace}{\mspace{4mu}}
\newcommand{\commaspace}{{,\smallspace}}
\DeclareRobustCommand{\makeset}[1]{%
    \@makelists{\opensetbracket}{\commaspace}{\closesetbracket}{#1}%
}
\DeclareRobustCommand{\makesetnocolor}[1]{%
    \@makelists{\{}{\commaspace}{\}}{#1}%
}
\DeclareRobustCommand{\tupp}[1]{%
    \@makelists{\langle}{\commaspace}{\rangle}{#1}%
}

\DeclareRobustCommand{\tup}[1]{%
        \@makelists{\mathopen{}\left\langle}{\commaspace}{\right\rangle\mathclose{}}{#1}%
}

\DeclareRobustCommand{\makecprod}[1]{%
    \@makelists{}{\cprod}{}{#1}%
}

\DeclareRobustCommand{\makecartprod}[1]{%
    \@makelists{}{\cartprod}{}{#1}%
}

\DeclareRobustCommand{\maketupconcat}[1]{%
    \@makelists{}{\tupconcat}{}{#1}%
}

\DeclareRobustCommand{\makesetBig}[1]{%
    \@makelists{{\mathopen{}\color{formulasetcolor}\mathbf{\Big\{}}}{\commaspace}{{\color{formulasetcolor}\mathbf{\Big\}}\mathclose{}}}{#1}%
}

\DeclareRobustCommand{\makesett}[1]{%
    \@makelists{\cleft[formulasetcolor]{\{}}{\commaspace}{\cright[formulasetcolor]{\}}}{#1}%
}

\DeclareRobustCommand{\makelist}[1]{%
    \@makelists{[}{\commaspace}{]}{#1}%
}

\DeclareRobustCommand{\cObj}[1]{%
    \@makelists{\fatstart}{\commaspace}{\fatend}{#1}%
}

\makeatother

\DeclarePairedDelimiter\mypars{\lparen}{\rparen}
\newcommand{\pars}[1]{\mypars*{#1}}

\edef\BlackColor{\csname\string\color@.\endcsname}

\makeatletter
\newcommand{\@CurrentColor}{}%
\DeclareRobustCommand{\ColorIfNowBlack}[2]{%
    \edef\@CurrentColor{\csname\string\color@.\endcsname}%
    \IfStrEq{\@CurrentColor}{\BlackColor}{%
        {\color{#1}#2}%
    }{%
        #2%
    }%
}
\makeatother

\newcommand{\boldtimes}{\mathbin{\textbf{×}}}

\newenvironment{forslides}{%
    \colorlet{shadecolor}{Pink!50}%
    \begin{shaded*}%
        \textbf{This content was created for slides. Do not erase or change the names}.\\%
        }{%
    \end{shaded*}}

\crefformat{equation}{(#2#1#3)}
\makeatletter
\def\mathcolor#1#{\@mathcolor{#1}}
\def\@mathcolor#1#2#3{%
  \protect\leavevmode
  \begingroup
    \color#1{#2}#3%
  \endgroup
}
\makeatother
\definecolor{norcolor}{RGB}{213,87,96}

\definecolor{GCQuestionsColor}{RGB}{213,87,96} %
\definecolor{GCAnswersColor}{RGB}{46,117,189} %
\definecolor{GCExpl}{named}{orange}
\definecolor{GCBool}{named}{black}
\definecolor{transmuted}{named}{black}

\newcommand{\theonlyQ}{\mathcolor{GCQuestionsColor}{\singletonel}}
\newcommand{\theonlyA}{\mathcolor{GCAnswersColor}{\singletonel}}
\newcommand{\singletonQ}{\mathcolor{GCQuestionsColor}{\singleton}}
\newcommand{\singletonA}{\mathcolor{GCAnswersColor}{\singleton}}

\renewcommand{\opel}{}

\newcommand{\pos}{\styleobj{\mathbf{p}}}
\newcommand{\vel}{\styleobj{\mathbf{v}}}

\newcommand{\GCname}{GC}
\newcommand{\Gname}{G}

\newcommand{\GCatname}{GCat}
\newcommand{\GC}{\CatSymbol{\GCname}}
\newcommand{\GG}{\CatSymbol{\Gname}}
\newcommand{\GSet}{\ensuremath{\GG\Set}\xspace}
\newcommand{\GCat}{\CatSymbol{\GCatname}}
\newcommand{\GSetB}{\GG(\Set,\CatB)}

\newcommand{\prcat}{\CatSymbol{\mathcolor{morphisms}{P}\mathcolor{norcolor}{N}}}

\newcommand{\theonlynor}{\mathcolor{norcolor}{\singletonel}}
\newcommand{\NomS}{\mathcolor{norcolor}{\mathrm{N}}}
\newcommand{\HomS}{\mathcolor{morphisms}{\mathrm{H}}}
\newcommand{\NomSetS}[3]{\NomS^{#1}_{#2#3}}
\newcommand{\HomSetS}[3]{\HomS^{#1}_{{#2}{#3}}}
\renewcommand{\Hom}{\mathcolor{morphisms}{\Homname}}
\renewcommand{\Nom}{\mathcolor{norcolor}{\mathrm{Nom}}}

\newcommand{\nmincompat}[2]{\mathcolor{GCExpl}{i}_{#1#2}}
\renewcommand{\enid}{{\gamma}}
\renewcommand{\enidof}[2]{\enid_{#1#2}}
\renewcommand{\enm}{{{\beta}}}
\renewcommand{\enmof}[4]{\enm_{#1#2#3#4}}
\renewcommand{\mtimescat}{\mathbin{{\pmb{\otimes}}}} %
\newcommand{\homobject}[3]{{\operatorname{\alpha}}_{#1{#2}{#3}}}
\newcommand{\homobjectset}[1]{{\operatorname{\alpha}}_{#1}}
\renewcommand{\funsp}{\EM{{\mathbf{F}}}} %
\renewcommand{\ressp}{\EM{{\mathbf{R}}}} %
\renewcommand{\adpa}{\stylemorph{{d}}} %
\renewcommand{\adpb}{\stylemorph{{e}}} %
\renewcommand{\F}[1]{{\mathcolor{black}{#1}}}
\renewcommand{\R}[1]{{\mathcolor{black}{#1}}}
\renewcommand{\stylepos}[1]{{\mathcolor{black}{#1}}}
\renewcommand{\stylesets}[1]{{\ColorIfNowBlack{black}{\mathbf{#1}}}}
\renewcommand{\interv}[2]{{[}#1,\smallspace#2{]}}

\definecolor{formulasetcolor}{named}{black}
\renewcommand{\manifold}[1]{\mathbb{M}#1}
\renewcommand{\stylemaps}[1]{{{#1}}}
\renewcommand{\styleelements}[1]{\mathcolor{black}{#1}}
\newcommand{\functhen}{\mathbin{{\mthensymbol}}\linebreak[0]} %
\renewcommand{\mthen}{\mathbin{{\mthensymbol}}\linebreak[0]} %

\renewcommand{\funcprod}{\mathbin{\mathcolor{black}{\times}}} %
\newcommand{\proofs}{\mathcolor{morphisms}{P}}
\newcommand{\refutations}{\mathcolor{norphisms}{R}}
\renewcommand{\stylesets}[1]{{\ColorIfNowBlack{formulasetcolor}{{#1}}}}
\renewcommand{\profto}{\stylemorph{\mathrel{\slashedrightarrow}}\linebreak[0]} %
\renewcommand{\leftunitor}{{\operatorname{\emph{lu}}}} %
\renewcommand{\rightunitor}{{\operatorname{\emph{ru}}}} %
\renewcommand{\associator}{{\operatorname{\emph{as}}}} %

\newcommand{\gcQ}{\mathcolor{GCQuestionsColor}{Q}} %
\newcommand{\gcA}{\mathcolor{GCAnswersColor}{A}} %
\newcommand{\gcR}{\mathcolor{GCExpl}{C}} %
\newcommand{\gcC}{\mathcolor{GCExpl}{\kappa}} %
\newcommand{\gcq}{\mathcolor{GCQuestionsColor}{q}} %
\newcommand{\gca}{\mathcolor{GCAnswersColor}{a}} %

\newcommand{\gcE}{\mathcolor{GCBool}{\mathbf{B}}} %
\newcommand{\gcmaname}{r}
\newcommand{\gcmbname}{s}

\newcommand{\gcm}[1]{\mathbf{#1}}
\newcommand{\gcmFdec}[1]{{#1}^{\mathcolor{GCAnswersColor}{\sharp}}}
\newcommand{\gcmBdec}[1]{{#1}_{\mathcolor{GCQuestionsColor}{\flat}}}
\newcommand{\gcmF}[1]{{\mathcolor{GCAnswersColor}{#1}}^{\mathcolor{GCAnswersColor}{\sharp}}}
\newcommand{\gcmB}[1]{\mathcolor{GCQuestionsColor}{#1}_{\mathcolor{GCQuestionsColor}{\flat}}}
\newcommand{\gcmE}[1]{\mathcolor{GCExpl}{#1}^{\mathcolor{GCExpl}{\ast}}}
\newcommand{\gcma}{\gcm{\gcmaname}}
\newcommand{\gcmaf}{\gcmF{\gcmaname}}
\newcommand{\gcmab}{\gcmB{\gcmaname}}
\newcommand{\gcmae}{\gcmE{\gcmaname}}
\newcommand{\gcmb}{\gcm{\gcmbname}}
\newcommand{\gcmbf}{\gcmF{\gcmbname}}
\newcommand{\gcmbb}{\gcmB{\gcmbname}}
\newcommand{\gcmbe}{\gcmE{\gcmbname}}

\newcommand{\gctrue}{\mathcolor{GCExpl}{\true}}
\newcommand{\gcfalse}{\mathcolor{GCExpl}{\false}}
\newcommand{\Nac}{{\NomSetS{}{\Obja}{\Objc}}}
\newcommand{\Nab}{{\NomSetS{}{\Obja}{\Objb}}}
\newcommand{\Nbc}{{\NomSetS{}{\Objb}{\Objc}}}
\newcommand{\Hac}{{\HomSetS{}{\Obja}{\Objc}}}
\newcommand{\Hab}{{\HomSetS{}{\Obja}{\Objb}}}
\newcommand{\Hbc}{{\HomSetS{}{\Objb}{\Objc}}}

\newcommand{\gcParName}{\llpar}
\newcommand{\gcPlusName}{\sqcup}
\newcommand{\gcPlus}{\mathbin{\boldsymbol{\gcPlusName}}}
\newcommand{\gcPlusob}{\mathbin{\obcomp{\gcPlus}}}
\newcommand{\gcPlusmor}{\mathbin{\morcomp{\gcPlus}}}
\newcommand{\gcCtimes}{\mathbin{\mathcolor{GCExpl}{\otimes}}}
\newcommand{\gcCplus}{\mathbin{\mathcolor{GCExpl}{\gcPlusName}}}
\newcommand{\gcCpar}{\mathbin{\mathcolor{GCExpl}{\gcParName}}}

\newcommand{\gcMehName}{\ast}
\newcommand{\gcMeh}{\mathbin{\boldsymbol{\gcMehName}}}
\newcommand{\gcCmeh}{\mathbin{\mathcolor{GCExpl}{\gcMehName}}}
\newcommand{\gcMehob}{\mathbin{\obcomp{\gcMehName}}}
\newcommand{\gcMehmor}{\mathbin{\morcomp{\gcMehName}}}

\def\llpar{\mathbin{\rotatebox[origin=c]{180}{\&}}}

\def\lltimes{\otimes}%

\def\llwith{\&}%

\def\llplus{\oplus}%

\def\llzero{\boldsymbol{0}}%
\def\llone{\boldsymbol{1}}%
\def\llbot{\boldsymbol{\bot}}%
\def\lltop{\boldsymbol{\top}}%

\newcommand{\theprod}{{\color{black}(\gcma \lltimesmor \gcmb)}}
\newcommand{\thesum}{{\color{black}(\gcma \gcPlusmor \gcmb)}}
\newcommand{\thepar}{{\color{black}(\gcma \llparmor \gcmb)}}
\newcommand{\theseries}{{\color{black}(\gcma \mthen \gcmb)}}
\newcommand{\themeh}{{\color{black}(\gcma \gcMehmor \gcmb)}}

\newcommand{\obcomp}[1]{{#1}}
\newcommand{\morcomp}[1]{{#1}}
\newcommand{\cprodob}[1]{\mathbin{{\times}_{#1}}}
\newcommand{\cprodmor}[1]{\mathbin{{\times}_{#1}}}
\newcommand{\csumob}[1]{\mathbin{{+}_{#1}}}
\newcommand{\csummor}[1]{\mathbin{{+}_{#1}}}
\newcommand{\llparob}{\mathbin{\obcomp{\llpar}}}
\newcommand{\llparmor}{\mathbin{\morcomp{\llpar}}}

\newcommand{\lltimesob}{\mathbin{\obcomp{\lltimes}}}
\newcommand{\lltimesmor}{\mathbin{\morcomp{\lltimes}}}

\renewcommand{\mtoin}[1]{\mathrel{{\to}_{#1}}\linebreak[0]}
\renewcommand{\mfromin}[1]{\mathrel{{\leftarrow}_{#1}}} %

\newcommand{\initobof}[1]{0_{#1}}
\newcommand{\termobof}[1]{1_{#1}}

\renewcommand{\cartprod}{\mathbin{\times}}
\renewcommand{\funcprod}{\mathbin{\morcomp{\times}}}
\newcommand{\thefloor}[1]{\aword{floor}(#1)}
\newcommand{\theround}[1]{\aword{round}(#1)}
\newcommand{\theceil}[1]{\aword{ceil}(#1)}

\presetkeys{todonotes}{disable}{}

\newcommand{\proceedings}[2]{\ifthenelse{\boolean{proceedings}}{#1}{#2}}

\newboolean{proceedings}
\setboolean{proceedings}{true}

\definecolor{transmuted}{RGB}{0,0,0}
\setboolean{statuscolors}{false}
\setboolean{debugimages}{false}
\setboolean{showslides}{false}
\setboolean{showproofs}{false}
\setboolean{cachepdf}{false}
\setboolean{devel}{false}
\setboolean{instructors}{false}
\newboolean{prooftext}
\setboolean{prooftext}{false}

\providecommand{\event}{ACT 2022} %
\usepackage{pdfpages}
\usepackage{breakurl}             %
\usepackage{underscore}           %
\usepackage{xstring}

\theoremstyle{definition}

\newlength{\mysep}
\newtheorem{theorem}{Theorem}
\newtheorem*{theorem*}{Theorem}

\newtheorem{proposition}[theorem]{Proposition}
\newtheorem*{proposition*}{Proposition}

\newtheorem*{lemma*}{Lemma}

\newtheorem{remark}[theorem]{Remark}

\newtheorem{example}[theorem]{Example}

\newtheorem{definition}[theorem]{Definition}
\newtheorem*{definition*}{Definition}

\newcommand\yqed[1]{%
    \leavevmode\unskip\penalty9999 \hbox{}\nobreak\hfill
    \quad\hbox{#1}}
    \newcommand\definitionend{\yqed{$\vartriangleleft$}}

\AtBeginDocument{
        \addtolength{\abovedisplayskip}{-0.4ex}
        \addtolength{\abovedisplayshortskip}{-0.2ex}
        \addtolength{\belowdisplayskip}{-0.4ex}
        \addtolength{\belowdisplayshortskip}{-0.2ex}
}

\title{Categorification of Negative Information using Enrichment}

\author{Andrea Censi \quad Emilio Frazzoli \quad Jonathan Lorand \quad Gioele Zardini
\institute{Institute for Dynamic Systems and Control\\
Department of Mechanical and Process Engineering\\
ETH Zurich, Switzerland}
\email{\{acensi,efrazzoli,jlorand,gzardini\}@ethz.ch}
}
\def\titlerunning{Categorification of Negative Information using Enrichment}
\def\authorrunning{Censi, Frazzoli, Lorand, Zardini}
\begin{document}
    \maketitle

    \begin{abstract}
        In many engineering applications it is useful to reason about ``negative information''.
        For example, in planning problems, providing an optimal solution is the same as giving a feasible solution (the ``positive'' information) together with a proof of the fact that there cannot be feasible solutions better than the one given (the ``negative'' information).
        We model negative information by introducing the concept of ``norphisms'', as opposed to the positive information of morphisms.
        A ``nategory'' is a category that has ``nom''-sets in addition to hom-sets, and specifies the interaction between norphisms and morphisms.
        In particular, we have composition rules of the form $\text{morphism} + \text{norphism} \to \text{norphism}$.
        Norphisms do not compose by themselves; rather, they use morphisms as catalysts.
        After providing several applied examples, we connect nategories to enriched categtory theory.
        Specifically, we prove that categories enriched in de Paiva's dialectica categories $\GC$, in the case $\CatC = \Set$ and equipped with a modified monoidal product, define nategories which satisfy additional regularity properties.
        This formalizes negative information categorically in a way that makes negative and positive morphisms equal citizens.
    \end{abstract}

    \section{Introduction}

\subsection{Manipulation of negative information is important in applications of category theory}

Our research group's background is in robotics and systems theory.
In these fields, we have found that category theory can describe well many of the structures in our problems, but something is often missing: we find ourselves in the position of reasoning and writing algorithms that manipulate ``negative information'', but we do not know what is an appropriate categorical concept for it.
We give some examples.

\textbf{Robot motion planning} can be formalized as the problem of finding a trajectory through an environment, respecting some constraint (e.g., avoiding obstacles).
One can think of the robot configuration manifold $\manifold$ as a category where the objects are elements of the tangent bundle and the morphisms are the feasible paths according to the problem constraints.
The output of planning problems has an intuitive representation in category theory, if the problem is feasible.
A \emph{path} planning algorithm is given two objects and must compute a \emph{morphism} as a solution.
A \emph{motion} planning algorithm would compute a trajectory, which could be seen as a \emph{functor} $F$ from the manifold $[0,T]$ to $M$ with $F(0)=A$ and $F(T)=B$.
However, if the problem is infeasible--if no morphisms between two points can be found---if the algorithm must present a \emph{certificate of infeasibility}--what is the equivalent concept in category theory?

In many cases, the problems are not binary (either a solution exists or not, either a proposition is true or not) but we care about the performance of solutions.
For example, consider the case of the \emph{weighted shortest path problem in dynamic programming}.
The problem is to find a path through a graph that minimizes the sum of the weights of the edges on the path.
In robotics, this can be used for planning problems, where the weights could represent the time, the distance, or the energy required by a robot to traverse an edge, and the nodes are either regions of space or, more generally, joint states of the world and environment.
Proving that a path is optimal means producing the path \emph{together with} a proof that there are no shorter paths.
This is called a ``certificate of optimality'' and, like certificates of infeasibility, is negative information as it consists in negating the existence of a certain class of paths.
Interestingly, one can see algorithms such as Dijkstra's algorithm as constructing both positive and negative information at the same time, such that when a path is finally found, we are sure that there are no shorter ones \cite{delling2009engineering}.

In some cases, the negative information is a first-class citizen which is critical to the efficiency.
Algorithms such as A* require the definition of \emph{heuristic} functions, which is negative information: they provide a \emph{lower bound} on the cost of a path between two points.
And better heuristics make the algorithm faster.
Again, we ask, what could be the categorical counterpart of heuristics?

In \emph{co-design} \cite{fong2018seven, censi2022}, a morphism $\funsp \mto \ressp$ describes what functionality can be achieved with which resources.
They are characterized as boolean profunctors, that is, monotone functions $\funsp\op \cartprod \ressp \to \Bool$.
The negative information would be a ``nesign'' problem that characterizes an impossibility.
For example, if $\funsp=\ressp=\text{Energy}$, we expect that in this universe we cannot find a realizable morphism $\adp$ that satisfies $\adp(2 J, 1 J)$ (obtaining 2 Joules from 1 Joule).
Can this be expressed as some sort of morphism? In which category does it live?

\subsection{Our approach: ``Categorification'' of negative information}

We briefly describe our thought process in finding a formalization for dealing with negative information.

One approach could have been to build structure on top of a category, at a higher level, using logic.
We eschew this approach because of the belief that we should find a duality between positive and negative information that puts them ``at the same level''.

Our approach has been one in the spirit of ``categorification'': representing the negative information with a concrete structure for which to find axioms and inference rules.

An early influence in our thinking was the paper of Shulman about ``proofs and refutations''~\cite{shulman2018affine}.
What follows is a simplified explanation of one of the concepts of the paper.
Consider a category where objects are propositions and morphisms $\Obja \to \Objb$ are propositions $\Obja\Rightarrow \Objb$
(with the particular case of $\Obja \simeq (\top \to \Obja$)).
We can then consider the type $\proofs(\Obja\to \Objb)$ of \emph{proofs} and the type $\refutations(\Obja\to \Objb)$ of \emph{refutations}, which correspond to \emph{positive} and \emph{negative} information.
According to intuitionist logic, $\proofs(\Obja\to \Objb) = (\proofs(\Obja)\to \proofs(\Objb)) \times (\refutations(\Objb)\to \refutations(\Obja))$: a proof of $\Obja\Rightarrow \Objb$ is a way to convert a proof of $\Obja$ into a proof of $\Objb$ together with a way to convert a refutation of $\Objb$ into a refutation of $\Obja$.

In that paper, proofs and refutations, positive and negative information, are treated \emph{at the same level} but not symmetrically---proof and refutations have different semantics, and $\proofs$ and $\refutations$ map products and coproducts ($\vee$, $\wedge$) to different linear logic operators.
This led to the idea that negative information should be at the same level of positive information: if positive information is represented by morphisms, then also the negative information should be described as ``negative arrows'' between objects, which we called \emph{norphisms} (for negative morphisms).

We also realized that the positive/negative information duality we are looking for is richer than the structure of proofs/refutations in logic.
In (classical/intuitionistic) logic, one expects the existence of either a proof of a proposition $A$, a refutation of $A$, or neither, but not both.
Instead, in our formalization, norphisms are a more general notion, which can coexist with morphisms and give complementary information, as in the planning examples in the introduction.

An initial idea was to consider for each category a ``twin'' category, whose morphisms would be the norphisms we were looking for to represent the negative information; however, this idea failed.
In the course of the paper, it will be clear that positive/negative information cannot be decoupled, because negative information cannot be composed independently of positive information. In the end, we unite them by viewing them as part of a single enrichement structure. 

\subsection{Plan of the paper}
This paper follows an {inductive exposition} and is divided in two parts.

In the \textbf{first part} we provide the \textbf{motivation and several examples of representing negative information with ``norphism'' structure}.
In Section \ref{sec:thin-cat-case} we consider the case of a thin category.
In this simple setting we can already see that norphisms compose differently than morphisms, and that we need two composition rules for them.
In Section \ref{sec:nategories} we state our main definition, that of a ``nategory'', and in Section \ref{sec:canonical-constructions} we show some canonical ways to build a nategory out of a category. 
In Sections \ref{sec:berg} and \ref{sec:dp} we discuss two examples, \Berg and \DP, which have norphism structures in which norphisms and morphisms are not mutually exclusive.

In the \textbf{second part} our goal is to provide \textbf{an elegant way to think of norphisms and their composition by using enriched category theory}.
By doing so, we show that the additional structure of norphisms and their composition, rules which might initially appear ``funky'', is not an arbitrary structure, but rather it is as ``natural'' as the positive information of morphisms. In Section \ref{sec:GSet} we introduce the dialectica category $\GSet$ \cite{de1989dialectica, blass1995questions} and define a monoidal product for it which is slightly different than the ones usually used as linear logic connectives. Then, in  Section \ref{sec:payoff}, we prove that $\GSet$-enriched categories encode nategories which satisfy some additional compatibilities between morphisms and norphisms. These additional compatibilities are not satisfied in certain examples of interest to us, therefore we have refrained from including them directly in our definition of nategory.

    \section{Building intuition: the case of thin categories}
\label{sec:thin-cat-case}
To build an intuition about norphisms, we look at the case of ``thin'' categories, in which each hom-set contains at most one morphism.
Thin categories are essentially pre-orders.
To aid the interpretation, one can think of a pre-order as defining a reachability relation, in which a morphism $\Obja \mto \Objb$ represents ``I can reach $\Objb$ from $\Obja$''.
Or, we can think of morphisms as (proof-irrelevant) implications: $\Obja \mto \Objb$ represents ``I can prove $\Objb$ from $\Obja$''.
In a thin category, negative information is limited to indicate the refutation of positive information.
Therefore, a norphism $\nora \colon \Obja \norto \Objb$ is equivalent to ``There are no morphisms from $\Obja$ to $\Objb$''.
Particularly, this means ``I cannot reach $\Objb$ from $\Obja$'' or ``I cannot prove $\Objb$ from $\Obja$''.

We will later see that, in general, norphisms need not necessarily be mutually exclusive with morphisms. Still, this example is sufficient to get us started in appreciating how morphisms and norphisms compose differently.
The composition rule for morphisms reads:
\begin{equation}
    \label{eq:composition-mor}
    \prfperiod{\mora\colon \Obja \mto \Objb}{\morb\colon \Objb \mto \Objc}
    {(\mora \mthen \morb)\colon \Obja \mto \Objc}
\end{equation}
Mimicking this, one could start with two norphisms $\nora\colon \Obja \norto \Objb$ and $\norb\colon \Objb \norto \Objc$ and expect to be able to say something about a norphism $\Obja \norto \Objc$, with a composition rule of the form:
\begin{equation}\label{eq:composition-nor-2}
    \prfperiod{\nora\colon \Obja \norto \Objb}{\norb\colon \Objb \norto \Objc}
    {\textcolor{norcolor}{???} \colon \Obja \norto \Objc}
\end{equation}
However, norphisms do not compose this way.
In fact, one can derive the following rule:
\begin{equation}
    \label{eq:nor-composition}
    \prfperiod{\norc \colon \Obja \norto \Objc}{\quad}{\Objb \colon \Ob_\CatC}{(\nora\colon \Obja \norto \Objb) \boolor (\norb\colon \Objb \norto \Objc)}
\end{equation}
This rule is ``the dual'' of (\ref{eq:composition-mor}) in the same sense as these two axioms are dual:
\begin{equation}\label{eq:nor-duality}
    \prfcomma{\true}{\Obja \mto \Obja} \qquad \prfcomma{\Obja \norto \Obja}{\false}
\end{equation}
that is, in the sense of flipping vertically and negating the propositions.

We read (\ref{eq:nor-composition}) as saying that if there is no morphism $\Obja \mto \Objc$, it is because, for every possible intermediate $\Objb$, there cannot be a morphism $\Obja \mto \Objb$ or $\Objb\mto \Objc$.
Note that composition goes in the ``opposite'' direction meaning that from one norphism, we get some information about the existence of one or two in a pair.
The composition (\ref{eq:nor-composition}) is not constructive: from the ``$\boolor$'', we do not know which side we can create.
Indeed, this composition highlights the asymmetry between morphisms and norphisms:
morphisms compose constructively by themselves (i.e., without taking into account norphisms);
norphisms, instead, do not ``compose'', but rather ``decompose'' by themselves.
To construct norphisms, we need to start from a norphism \emph{and} a morphism that acts as a ``catalyst''.
When interpreting a thin category as a graph, if there is a norphism $\nora \colon \Obja \norto \Objb$, it means that for any $\Objb$, the path $\Obja \mto \Objb \mto \Objc$ must be interrupted in one part or the other, because otherwise we would have a contradiction.
Indeed, if we know that morphisms $\mora\colon \Obja \mto \Objb$ and $\morb\colon \Objb \mto \Objc$ exist, then their composition $\mora\mthen \morb\colon \Obja \mto \Objc$ must exist, and therefore no norphism $\nora\colon \Obja \norto \Objc$ can exist.
This observation can be turned around in a constructive way.
Starting from a morphism $\mora\colon \Obja \mto \Objb$ and a norphism $\nora\colon \Obja \norto \Objc$ (i.e., morphisms and norphisms with the same source), we can infer a norphism $\northenb{\mora}{\nora}\colon \Objb\norto \Objc$ (i.e., there cannot be a morphism $\Objb\mto \Objc$):
\todo[inline]{Check norphism symbol, as in the pic it has to be done differently.
}
\begin{equation}
    \label{eq:rule-same-source}
    \begin{tabular}{cc}
        \includesag{nor-compose-const-a}\qquad \qquad &
        \begin{tikzpicture}
            \node at (0,2){\prfperiod{\Objb \overset{\mora}{\stylemorph{\leftarrow}} \Obja \overset{\nora}{\norto}\Objc}{\Objb \overset{\northenb{\mora}{\nora}}{\norto} \Objc}
            };
        \end{tikzpicture}
    \end{tabular}
\end{equation}

Symmetrically, starting from a morphism $\morb\colon \Objb \mto \Objc$ and a norphism $\nora\colon \Obja \norto \Objc$ (i.e., morphisms and norphisms with the same target), we can infer a norphism $\northena{\nora}{\mora} \colon \Obja\norto \Objb$:
\todo[inline]{check symbol in pic if changing macro}
\begin{equation}
    \label{eq:rule-same-target}
    \begin{tabular}{cc}
        \includesag{nor-compose-const-b}\qquad \qquad &
        \begin{tikzpicture}
            \node at(0,2){\prfperiod{\Obja \overset{\nora}{\norto} \Objc \overset{\morb}{\stylemorph{\leftarrow}}\Objb}{\Obja \overset{\northena{\nora}{\morb}}{\norto} \Objb}};
        \end{tikzpicture}
    \end{tabular}
\end{equation}
\showslides{
    \begin{forslides}
        \begin{equation}\label{eq:norcomp1}
            \prf{
                \Objb \overset{\mora}{\stylemorph{\leftarrow}} \Obja \overset{\nora}{\norto}\Objc
            }{
                \Objb \overset{\mora \,\northenbsymb\, \nora}{\norto} \Objc
            }
        \end{equation}

        \begin{equation}\label{eq:norcomp2}
            \prf{
                \Obja \overset{\nora}{\norto} \Objc \overset{\morb}{\stylemorph{\leftarrow}}\Objb
            }{
                \Obja \overset{\nora \,\northenasymb\, \morb}{\norto} \Objb
            }
        \end{equation}
    \end{forslides}
}

Note that the new norphism is pointing in the ``same direction'' as the starting one, meaning that either source or target are preserved.

\section{Describing negative information: \emph{nategories}}
\label{sec:nategories}

In this section we make the notion of norphisms more precise, by defining the additional structure which a category must have in order to encode negative information.

\begin{definition}[Nategory]\label{def:nategory}
    A locally small \emph{nategory} \CatC is a locally small category with the following additional structure.
    For each pair of objects $\Obja,\Objb \setin \Ob_{\CatC}$, in addition to the set of morphisms $\HomSet{\CatC}{\Obja}{\Objb}$, we also specify:
    \begin{compactitem}
        \item A set of norphisms $\NomSet{\CatC}{\Obja}{\Objb}$.
        \item An \emph{incompatibility relation}, which we write as a binary function
              \begin{equation} \label{eq:010-nategory-incompat}
                  \nmincompat{\Obja}{\Objb}\colon\NomSet{\CatC}{\Obja}{\Objb} \cartprod \HomSet{\CatC}{\Obja}{\Objb}   \to \makeset{\false, \true}.
              \end{equation}
    \end{compactitem}
    For all triples $\Obja,\Objb,\Objc$, in addition
    to the morphism composition function
    \begin{equation} \label{eq:011-nategory-comp}
        \mthen_{\Obja\Objb\Objc} \colon \HomSet{\CatC}{\Obja}{\Objb} \cartprod \HomSet{\CatC}{\Objb}{\Objc} \to \HomSet{\CatC}{\Obja}{\Objc},
    \end{equation}
    we require the existence of two norphism composition functions
    \begin{equation}\label{eq:012-nategory-norcomp}
        \begin{aligned}
            \northenbsymb_{\Obja\Objb\Objc} & \colon \HomSet{\CatC}{\Obja}{\Objb} \cartprod \NomSet{\CatC}{\Obja}{\Objc} \to \NomSet{\CatC}{\Objb}{\Objc},
            \\
            \northenasymb_{\Obja\Objb\Objc} & \colon \NomSet{\CatC}{\Obja}{\Objc} \cartprod \HomSet{\CatC}{\Objb}{\Objc} \to \NomSet{\CatC}{\Obja}{\Objb},
        \end{aligned}
    \end{equation}
    and we ask that they satisfy two ``equivariance'' conditions:
    \begin{align}
        \nmincompat{\Objb}{\Objc}(\northenb{\mora}{\nora}, \morb)
         & \Rightarrow
        \nmincompat{\Obja}{\Objc} ( \nora, \morab )\label{eq:cond1}\tag{equiv-1}, \\
        \nmincompat{\Obja}{\Objb} (\northena{\nora}{\morb}, \mora)
         & \Rightarrow
        \nmincompat{\Obja}{\Objc} ( \nora, \morab )\tag{equiv-2}\label{eq:cond2}.
    \end{align}
\end{definition}
We write $\nora : \Obja \norto \Objb$ to say that $\nora \in \NomSet{\CatC}{\Obja}{\Objb}$.
\begin{definition}[Exact nategory]\label{def:nategory-exact}
    If the two conditions (\ref{eq:cond1}) and (\ref{eq:cond2}) are satisfied with ``$\Leftrightarrow$'' instead of just ``$\Rightarrow$'', we say that the nategory is \emph{exact}.
\end{definition}

Condition (\ref{eq:cond1}) says that the norphism $\northenb{\mora}{\nora}$ can exclude the morphism $\morb$ only if $\morab$ is excluded by $\nora$.
The idea is that such a $\morb$ should not be excluded for any ``additional reasons'', but only on the grounds that $\morab$ is excluded by $\nora$.

\begin{figure}[tb]
    \begin{center}
        \begin{subfigure}[b]{0.48\linewidth}
            \includesag{syst-nor2-a}
            \caption{Starting situation.}
            \label{fig:syst-nor2-a}
        \end{subfigure}
        \begin{subfigure}[b]{0.48\linewidth}
            \includesag{syst-nor2-b}
            \caption{Find incompatibility with $\nora$.}
            \label{fig:syst-nor2-b}
        \end{subfigure}
        \begin{subfigure}[b]{0.48\linewidth}
            \includesag{syst-nor2-c}
            \caption{Pulling back incompatible morphisms.}
            \label{fig:syst-nor2-c}
        \end{subfigure}
        \begin{subfigure}[b]{0.48\linewidth}
            \includesag{syst-nor2-d}
            \caption{$\nomtohom{\Objb}{\Objc}(\northenb{\mora}{\nora})
    \subseteq
    \precomp{\mora}^{-1}(\nomtohom{\Obja}{\Objc}(\nora))$.}
            \label{fig:syst-nor2-d}
        \end{subfigure}
    \end{center}
    \vspace{-0.5cm}
    \caption{
        \label{fig:syst-nor2}
    }
\end{figure}

We draw some figures to develop further intuition (Fig. \ref{fig:syst-nor2}).
Let $\nomtohom{\Obja}{\Objb}$ denote the function which maps a norphism to the set of morphisms with which it is incompatible:
\begin{equation}
    \label{eq:nom-to-hom}
    \defmapperiod{
        \nomtohom{\Obja}{\Objb}
    }{
        \NomSet{\CatC}{\Obja}{\Objb}
    }{
        \to
    }{
        \powerset(\HomSet{\CatC}{\Obja}{\Objb})
    }{
        \nora
    }{
        \makeset{
            \mora \setin \HomSet{\CatC}{\Obja}{\Objb}\colon \nmincompat{\Obja}{\Objb}(\nora, \mora)
        }
    }
\end{equation}
We start in~Fig. \ref{fig:syst-nor2-a} with a norphism $\nora: \Obja \norto \Objc$ and
a morphism $\mora: \Obja \mto \Objb$.
In Fig. \ref{fig:syst-nor2-b} we apply $\nomtohom{\Obja}{\Objc}$ to find the set of incompatible morphisms $\nomtohom{\Obja}{\Objc}(\nora)$.
In Fig. \ref{fig:syst-nor2-c} we use the precomposition map
\begin{equation}\label{eq:020-precomp-map}
    \defmapcomma{
        \precomp{\mora}
    }{
        \HomSet{\CatC}{\Objb}{\Objc}
    }{
        \to
    }{
        \HomSet{\CatC}{\Obja}{\Objc}
    }{
        \morb
    }{
        \mora \mthen \morb
    }
\end{equation}
to obtain the set of morphisms
\begin{equation}
    \precomp{\mora}^{-1}(\nomtohom{\Obja}{\Objc}(\nora)).
\end{equation}
These are to be prohibited because when pre-composed with $\mora$
they give a morphism that is forbidden by $\nora$.
Now, in principle, it could be that our norphism inference is so powerful that
$\northenb{\mora}{\nora}$ manages to exclude all of these:
\begin{equation}
    \nomtohom{\Objb}{\Objc}(\northenb{\mora}{\nora}) = \precomp{\mora}^{-1}(\nomtohom{\Obja}{\Objc}(\nora)).
\end{equation}
In general, we are happy with the composition operation if it excludes part of those (but not more):
\begin{equation}\label{eq:inclusion-condition1}
    \nomtohom{\Objb}{\Objc}(\northenb{\mora}{\nora})
    \subseteq
    \precomp{\mora}^{-1}(\nomtohom{\Obja}{\Objc}(\nora)).
\end{equation}
It can readily be checked that (\ref{eq:inclusion-condition1}) is equivalent to (\ref{eq:cond1}).
Similarly, (\ref{eq:cond2}) is equivalent to requiring
\begin{equation}\label{eq:inclusion-condition2}
    \nomtohom{\Objb}{\Objc}(\northena{\nora}{\morb})
    \subseteq
    \postcomp{\morb}^{-1}(\nomtohom{\Obja}{\Objc}(\nora)),
\end{equation}
where $\postcomp{\morb}$ is the map ``post-composition with $\morb$''.

\proceedings{}{
    It is an inclusion of two sets $S_1 \subseteq S_2$.
    Call the generic element of the two sets $\morb$.
    Then inclusion means that the indicator function of the first implies the indicator function of the second: $S_1(\morb) \Rightarrow S_2(\morb)$.
    The indicator function of the first set is:
    \begin{equation}
        S_1(\morb) = \morb \in \nomtohom{\Objb}{\Objc}(\northenb{\mora}{\nora})
        = \nmincompat{\Objb}{\Objc}(\northenb{\mora}{\nora}, \morb).
    \end{equation}
    For the second set it is:
    \begin{equation}
        S_2(\morb) = \morb \in \precomp{\mora}^{-1}(\nomtohom{\Obja}{\Objc}(\nora))
        = \nmincompat{\Obja}{\Objc} ( \nora, \morab ).
    \end{equation}
}

\showslides{
    \begin{forslides}
        \begin{align}\label{eq:cond12bis}
            \nmincompat{\Objb}{\Objc}(\northenb{\mora}{\nora}, \morb)
             & \Rightarrow
            \nmincompat{\Obja}{\Objc} ( \nora, \morab ), \\
            \nmincompat{\Obja}{\Objb} (\nora \northenasymb \morb, \mora)
             & \Rightarrow
            \nmincompat{\Obja}{\Objc} ( \nora, \morab ).
        \end{align}
        \begin{align}\label{eq:cond12ter}
            \nmincompat{\Objb}{\Objc}(\northenb{\mora}{\nora}, \morb)
             & \Leftrightarrow
            \nmincompat{\Obja}{\Objc} ( \nora, \morab ), \\
            \nmincompat{\Obja}{\Objb} (\nora \northenasymb \morb, \mora)
             & \Leftrightarrow
            \nmincompat{\Obja}{\Objc} ( \nora, \morab ).
        \end{align}
    \end{forslides}
}
    \ifthenelse{\boolean{proceedings}}{}{
        \clearpage}
    \section{Canonical nategory constructions}\label{sec:canonical-constructions}

Here are three canonical constructions that allow us to get a nategory out of a category in a more or less straightforward way:
\begin{compactenum}
    \item Setting the norphism sets to be empty (Example \ref{exa:nom-empty});
    \item Setting the norphism sets to be singletons that negate the entire respective hom-sets (Example \ref{exa:nom-all});
    \item Setting the norphism sets to be the powerset of the respective hom-sets (Example \ref{exa:nor-subsets}).
\end{compactenum}

\begin{example}[A nategory with no norphisms]
    \label{exa:nom-empty}
    For any category~\CatC, let
    \begin{equation}\label{eq:010-nonorph}
        \NomSet{\CatC}{\Obja}{\Objb}\definedas\emptyset.
    \end{equation}
    For all pairs $\Obja,\Objb$ the
    function $\nmincompat{\Obja}{\Objb}$ is uniquely defined as it has an empty domain.
    The functions $\northenbsymb, \northenasymb$ also have empty domains.
    The conditions (\ref{eq:cond1}) and (\ref{eq:cond2}) are trivially verified.
    A nategory with no norphisms is just a category.
\end{example}

\begin{example}[Singleton norphism sets negating all morphisms]
    \label{exa:nom-all}
    In this construction, we turn a category into a nategory by making the choice that a norphism is a witness for the fact that the corresponding hom-set is empty.
    For any category~\CatC, let
    \begin{equation}\label{eq:020-one}
        \NomSet{\CatC}{\Obja}{\Objb}\definedas\makeset{\theonlynor},
    \end{equation}
    and
    for any pair $\Obja,\Objb$ and any $\mora\colon\Obja\mto\Objb$ let
    \begin{equation}\label{eq:021-one-true}
        \nmincompat{\Obja}{\Objb} ( \theonlynor, \mora) = \true.
    \end{equation}
    In this case, the element $\theonlynor$ is a witness for ``$\HomSet{\CatC}{\Obja}{\Objb}$ is empty''.
Next, we need to define the two maps:
    \begin{align}
        \northenbsymb & \colon \HomSet{\CatC}{\Obja}{\Objb} \cartprod \NomSet{\CatC}{\Obja}{\Objc} \to \NomSet{\CatC}{\Objb}{\Objc},
        \\
        \northenasymb & \colon \NomSet{\CatC}{\Obja}{\Objc} \cartprod \HomSet{\CatC}{\Objb}{\Objc} \to \NomSet{\CatC}{\Obja}{\Objb}.
    \end{align}
    The choice is forced, as there is only one norphism in the codomains.
    We obtain:
    \begin{align}\label{eq:022-one-compose}
        \northenb{\mora}{\theonlynor} & = \theonlynor, \\
        \northena{\theonlynor}{\morb} & = \theonlynor.
    \end{align}
    The conditions (\ref{eq:cond1}) and (\ref{eq:cond2}) are easily verified because $\nmincompat{\Obja}{\Objb}$ always evaluates to $\true$.

\end{example}

\begin{example}[Norphism sets are subsets of hom-sets]
    \label{exa:nor-subsets}

    For any category~\CatC, let
    \begin{equation}\label{eq:030-powerset}
        \NomSet{\CatC}{\Obja}{\Objb}=\powerset(\HomSet{\CatC}{\Obja}{\Objb}).
    \end{equation}
    Set the incompatibility relation as
    \begin{equation}\label{eq:030-powerset-incomp}
        \nmincompat{\Obja}{\Objb}(\nora, \mora) = \mora \in \nora.
    \end{equation}
    Define the composition operations as
    \begin{align}\label{eq:030-powerset-compose}
        \northenb{\mora}{\nora} & = \precomp{\mora}^{-1}(\nora), \\
        \northena{\nora}{\morb} & = \postcomp{\morb}^{-1}(\nora),
    \end{align}
    where $\precomp{\mora}$ and $\postcomp{\morb}$ are the pre- and post-composition maps
    \begin{equation}\label{eq:030-powerset-precomp}
        \defmapcomma{
            \precomp{\mora}
        }{
            \HomSet{\CatC}{\Objb}{\Objc}
        }{
            \to
        }{
            \HomSet{\CatC}{\Obja}{\Objc}
        }{
            \morb
        }{
            \mora \mthen \morb
        }
    \end{equation}
    \begin{equation}\label{eq:030-powerset-postcomp}
        \defmapperiod{
            \postcomp{\morb}
        }{
            \HomSet{\CatC}{\Obja}{\Objb}
        }{
            \to
        }{
            \HomSet{\CatC}{\Obja}{\Objc}
        }{
            \mora
        }{
            \mora \mthen \morb
        }
    \end{equation}
    Let's check the condition (\ref{eq:cond1}): $  \nmincompat{\Objb}{\Objc}(\northenb{\mora}{\nora}, \morb)
        \Rightarrow
        \nmincompat{\Obja}{\Objc} ( \nora, \morab )$.
        Using our definitions, we get
    \begin{equation}
        \morb \in \northenb{\mora}{\nora}
        \quad
        \Rightarrow
        \quad
        \morab \in \nora.
    \end{equation}
    Expanding the left-hand side we find
    \begin{align}
        \morb \in \precomp{\mora}^{-1}(\nora)
         & \quad
        \Rightarrow
        \quad
        \morab \in \nora.
    \end{align}
    Another expansion shows that both sides are the same:
    \begin{align}
        \morab \in \nora
         & \quad
        \Rightarrow
        \quad
        \morab \in \nora .
    \end{align}
    Checking condition (\ref{eq:cond2}) is analogous.
    Note that this nategory is exact.

\end{example}

Finally, we provide an example of a nategory that we will use later as a counter-example.

\begin{example}[Very weak composition operations]
    \label{exa:weak}
    For any category~\CatC, as in the previous example, use subsets of morphisms as the norphisms
    \begin{equation}
        \NomSet{\CatC}{\Obja}{\Objb}=\powerset(\HomSet{\CatC}{\Obja}{\Objb}),
    \end{equation}
    and set the incompatibility relation as
    \begin{equation}
        \nmincompat{\Obja}{\Objb}(\nora, \mora) = \mora \in \nora.
    \end{equation}
    However, define the composition operations as
    \begin{align}\label{eq:composition-weak}
        \northenb{\mora}{\nora} & = \emptyset, \\
        \northena{\nora}{\morb} & = \emptyset.
    \end{align}
    The equivariance conditions are still satisfied.
    For example condition (\ref{eq:cond1}),
    \begin{equation}
        \nmincompat{\Objb}{\Objc}(\northenb{\mora}{\nora}, \morb)
        \Rightarrow
        \nmincompat{\Obja}{\Objc} ( \nora, \morab ),
    \end{equation}
    becomes
    \begin{equation}
        \morb \in \emptyset
        \ \Rightarrow \
        \morab \in \emptyset,
    \end{equation}
    which is vacuously satisfied, because the premise is always false.
\end{example}

    \section{Example: hiking on the Swiss mountains}
\label{sec:berg}

In this section we present an example of planning, giving a more concrete description of the path planning problems mentioned in the introduction.
We describe $\Berg$, a category whose morphisms are hiking paths of various difficulty on a mountain.
We then consider the problem of finding paths of minimum length.

\begin{definition}[\Berg]\label{def:Berg}
    Let $\mapelevation \colon \reals^2\to \reals$ be a $C^1$ function, describing the elevation of a mountain.
    The set with elements $\tup{\setAel, \setBel, \mapelevation(\setAel,\setBel)}$ is a manifold $\manifold$ that is embedded in $\reals^3$.
    Let $\steepness=\interv{\steepnessL}{\steepnessU}\setsubset \reals$ be a closed interval of real numbers.
    The category $\Berg_{\mapelevation,\steepness}$ is specified as follows:
    \begin{compactenum}
        \item An object $\Obja$ is a pair $\tup{\pos,\vel}\setin \tangbundle \manifold$, where $\pos = \tupp{\pos_x,\pos_y,\pos_z}$ is the position, $\vel$ is the velocity, and $\tangbundle\manifold$ is the tangent bundle of the manifold.
        \item Morphisms are $C^1$ paths $\mora \colon [0, \tau] \to \manifold$ on the manifold satisfying evident boundary conditions (here $\tau \in \reals$ may vary). We also define, formally by decree, that for each $\tup{\pos,\vel}$ there is a \emph{trivial path} ``$[0,0] \to \manifold$'' which is $C^1$, has trace $\pos$, and velocity $\vel$.  %

              At each point $\pos = \mora(t)$ of a path we define the \emph{steepness} via the formula
              \begin{equation}\label{eq:Berg-steepness}
                  \steepfun(\tup{\pos,\vel}) \definedas \vel_z/\sqrt{\vel_x^2+\vel_y^2},
              \end{equation}
              where $\vel = \tfrac{d}{dt}\mora(t)$.
              We choose as morphisms only the paths that have the steepness values contained in the interval $\steepness$:
              \begin{equation}\label{eq:Berg-Hom}
                  \HomSet{\Berg_{\mapelevation,\steepness}}{\Obja}{\Objb}=\{\mora \text{ is a } C^1 \text{ path from }\Obja \text{ to }\Objb \text{ and } \steepfun(\mora)\setsubseteq \steepness\},
              \end{equation}
        \item Morphism composition is given by concatenation of paths.
        \item Identity morphisms are given by trivial paths. %
    \end{compactenum}
\end{definition}
\begin{figure}[h]
    \begin{center}
        \begin{tikzpicture}
            \node at (0,0) {\includegraphics[width=0.4\linewidth]{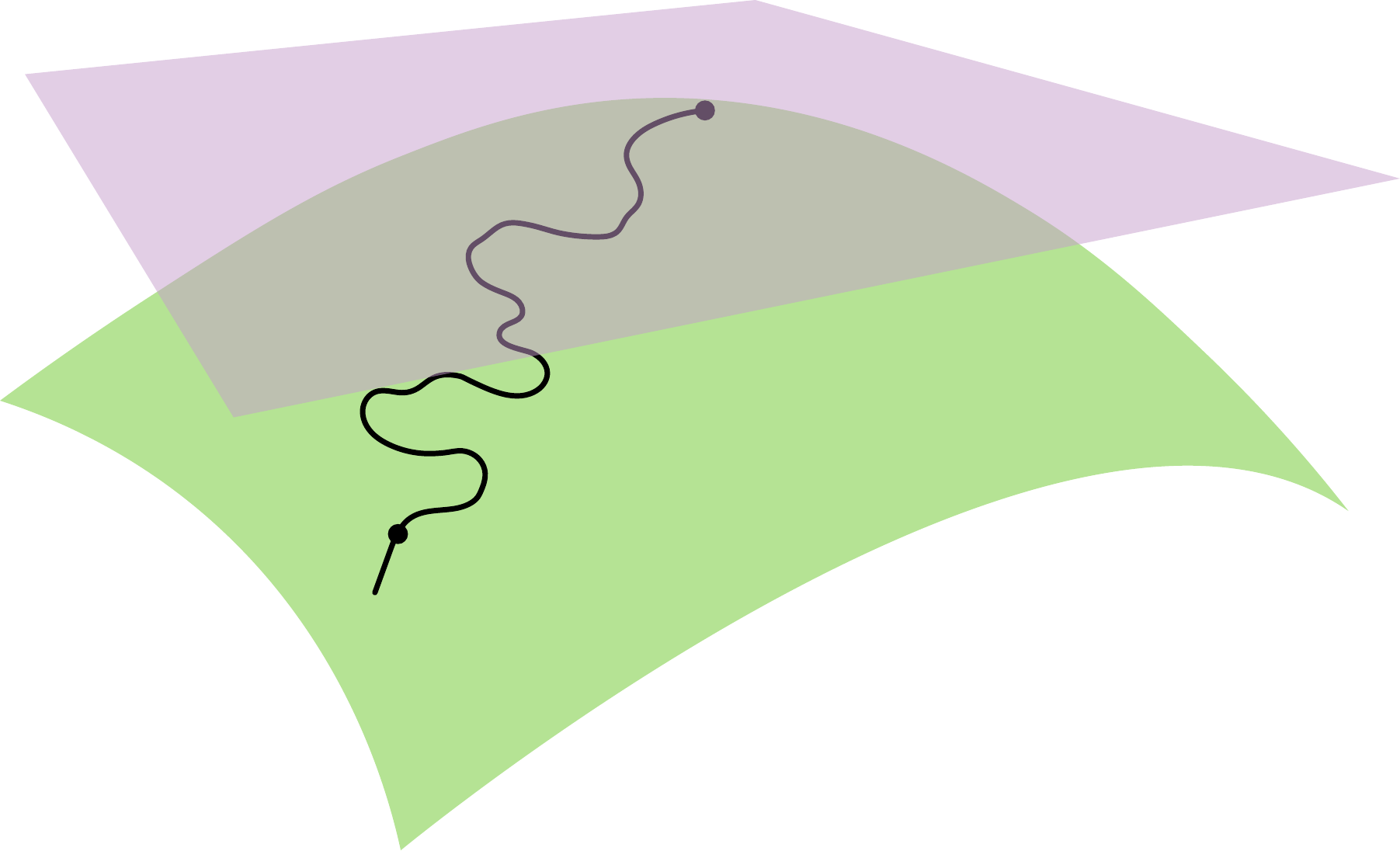}};
            \node at (2,0.1) {$\manifold$};
            \node at (-2.3,1.3){$\tangbundle\manifold$};
            \node at (-1.6,-0.5){$\Obja$};
            \node at (-0.3,1.6){$\Objb$};
            \draw[thick, -Triangle] (0,1.45)--node[pos=0.4,above]{$\vel$} (0.7,1.6){};
        \end{tikzpicture}
    \end{center}
\end{figure}

\ifthenelse{\boolean{proceedings}}{}{For the complete proof that $\Berg$ is a category, we refer the reader to Lemma \ref{lem:bergcat}.}

The steepness interval $\steepness$ allows considering different categories on the same mountain, with possible hikes varying in difficulty, measured via minimum/maximum steepness.
For example, a good hiker can handle $\steepness=\interv{-0.57}{0.57}$ (positive/negative $30^\circ$ slope).
If $\steepness=\interv{-0.57}{0}$, we are only allowed to climb down.
If $\steepness=\interv{0}{0}$, we can only walk along isoclines.

\paragraph{Interpretation of norphisms in $\Berg$}
What might a norphism be in this case?

One possibility is to let a norphism $\nora\colon \Obja \norto \Objb$ mean ``there exists no path from $\Obja$ to $\Objb$''.
This is a simple choice that is similar to Example \ref{exa:nom-all} and that makes morphisms and norphisms mutually exclusive.

We can obtain a more useful theory by letting norphisms carry information that is \emph{complementary} to morphisms
by interpreting them as \emph{lower bounds} on distances.
To see how this can work, let the set of norphisms be the real numbers completed by positive infinity:
\begin{equation}\label{eq:nom-nonneg}
    \NomSet{\Berg}{\Obja}{\Objb} \ \definedas \ \nonNegReals \setunion \makeset{+\infty}.
\end{equation}
Let $\lengthfun(\mora)$ be the length of the path (according to the manifold metric).
Then we interpret a norphism $\nora\colon \Obja \norto \Objb$ as a witness of ``for all paths $\mora\colon \Obja \mto \Objb$, we have $\lengthfun(\mora)\geq \nora$''.
The case $\nora = \infty$ negates any path from $\Obja$ to $\Objb$.
The incompatibility relation $\nmincompat{\Obja}{\Objb}$ can be written as follows:
\begin{equation}\label{eq:nom-nonneg-incompat}
    \nmincompat{\Obja}{\Objb}(\nora, \mora) = \lengthfun(\mora) < \nora.
\end{equation}
To say that a path $\mora$ is optimal means saying that $\mora$ is feasible \emph{and} that $\lengthfun(\mora)$ is a norphism:
\begin{equation}\label{eq:nom-nonneg-optimal}
    \prfdoubleperiod{\mora\colon \Obja \mto \Objb}{\stylenorph{\lengthfunnor(\mora)}\colon \Obja \norto \Objb}{\mora \text{\ is optimal}}
\end{equation}

\paragraph{Composition rules for norphisms}
Next, we define the following two composition rules
\begin{equation}
    \label{eq:berg-compose}
    \begin{aligned}
        \northenb{\mora}{\nora} & = \max\{\nora - \lengthfun(\mora), 0\}, \\
        \northena{\nora}{\morb} & = \max\{\nora - \lengthfun(\morb), 0\},
    \end{aligned}
\end{equation}
which are the equivalent of (\ref{eq:rule-same-source}) and (\ref{eq:rule-same-target}). See ~Fig. \ref{fig:composition-mnor-paths-lengths}.
Our reasoning is as follows. If for example $\mora$ is a path from $\Obja$ to $\Objb$, and we know that going from $\Obja$ to $\Objc$ has a distance of at least $\nora$, then any path from $\Objb$ to $\Objc$ must be at least $\nora - \lengthfun(\mora)$ long.
In this case,
\begin{equation}
    \begin{aligned}
        \nomtohom{\Obja}{\Objc}(\northenb{\mora}{\nora}) & =\{\morb\colon \lengthfun(\morb) < \max\{\nora-\lengthfun(\mora),0\}\}.
    \end{aligned}
\end{equation}
If $\nora < \lengthfun(\mora)$, then $\nomtohom{\Obja}{\Objc}(\northenb{\mora}{\nora})$ is empty, which differs from
\begin{equation}
    \precomp{\mora}^{-1}(\nomtohom{\Obja}{\Objb}(\nora))=\{\morb\colon \lengthfun(\morb)+\lengthfun(\mora)< \nora\}.
\end{equation}
The nategory is not exact.
However, since
\begin{equation}
    \{\morb\colon \lengthfun(\morb) < \max\{\nora-\lengthfun(\mora),0\}\}\subseteq \{\morb\colon \lengthfun(\morb)+\lengthfun(\mora)< \nora\},
\end{equation}
the nategory satisfies (\ref{eq:cond1}).
The check for (\ref{eq:cond2}) is analogous.

\begin{figure}[tb]
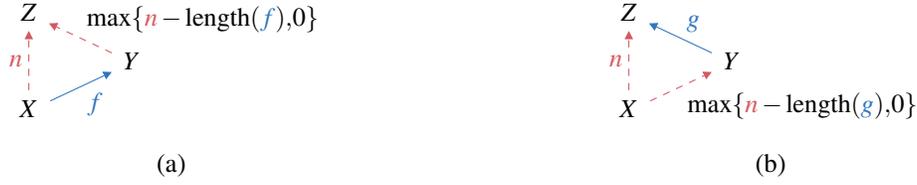

    \begin{center}
        \begin{subfigure}[b]{0.48\textwidth}
            \centering
            \includesag{berg-paths-lengths-b}
            \subcaption{
            }
            \label{fig:nor-comp-target-paths}
        \end{subfigure}
        ~
        \begin{subfigure}[b]{0.48\textwidth}
            \centering
            \includesag{berg-paths-lengths-a}
            \subcaption{
            }
            \label{fig:nor-comp-source-paths}
        \end{subfigure}
    \end{center}
    \vspace{-0.5cm}
    \caption{Composition of morphisms and norphisms in the case of paths and lengths.}
    \label{fig:composition-mnor-paths-lengths}
\end{figure}

\begin{example}\label{ex:berg-exact-version}
As a variant of the above, if we set
\begin{equation}\label{eq:nom-reals}
    \NomSet{\Berg}{\Obja}{\Objb} \ \definedas \ \reals \setunion \makeset{+\infty},
\end{equation}
and define the composition operations as
\begin{equation}\label{eq:nom-reals-compose}
    \begin{aligned}
        \northenb{\mora}{\nora} & = \nora - \lengthfun(\mora), \\
        \northena{\nora}{\morb} & = \nora - \lengthfun(\morb),
    \end{aligned}
\end{equation}
then the nategory is exact.
Indeed, for this case one has
\begin{equation}
    \begin{aligned}
        \nomtohom{\Obja}{\Objc}(\northenb{\mora}{\nora}) & =\{\morb\colon \lengthfun(\morb) < \nora-\lengthfun(\mora)\}, \\
                                                         & =\{\morb\colon \lengthfun(\morb)+\lengthfun(\mora)< \nora\} \\
                                                         & =\precomp{\mora}^{-1}(\nomtohom{\Obja}{\Objb}(\nora)).
    \end{aligned}
\end{equation}
\end{example}

\begin{example}\label{ex:berg-integer-version}
We may also think of a variation in which the norphisms are integers:
\begin{equation}\label{eq:nom-integers}
    \NomSet{\Berg}{\Obja}{\Objb} \ \definedas \ \mathbb{Z} \setunion \makeset{+\infty}.
\end{equation}
In this case we are limited to express constraints of the type
\begin{align}
    \lengthfun(\mora) \geq 0, \ \lengthfun(\mora) \geq 1, \ \lengthfun(\mora) \geq 2, \dots
\end{align}
We then define the composition rules as
\begin{equation}
    \begin{aligned}
        \label{eq:berg-comp-bis}
        \northenb{\mora}{\nora} & = \thefloor{\nora - \lengthfun(\mora)}, \\
        \northena{\nora}{\morb} & = \thefloor{\nora - \lengthfun(\morb)}.
    \end{aligned}
\end{equation}
\showslides{
    \begin{forslides}
        \begin{equation}
            \begin{aligned}
                \label{eq:berg-comp-round}
                \northenb{\mora}{\nora} & = \theround{\nora - \lengthfun(\mora)} \\
                \northena{\nora}{\morb} & = \theround{\nora - \lengthfun(\morb)}
            \end{aligned}
        \end{equation}
    \end{forslides}
}
In this case, (\ref{eq:cond1}) is satisfied, however our nategory is not exact, because in general
\begin{equation}
    \begin{aligned}
        \nomtohom{\Obja}{\Objc}(\northenb{\mora}{\nora}) & =\{\morb\colon \lengthfun(\morb) < \thefloor{\nora-\lengthfun(\mora)}\} \\
                                                         & \subsetneq \{\morb \colon \lengthfun(\morb)< \nora -\lengthfun(\mora)\} \\
                                                         & =\precomp{\mora}^{-1}(\nomtohom{\Obja}{\Objb}(\nora)),
    \end{aligned}
\end{equation}
since $\thefloor{\nora-\lengthfun(\mora)} +\lengthfun(\mora) \leq \nora$.
An analogous reasoning applies to (\ref{eq:cond2}). We note that using $\theround{-}$ or $\theceil{-}$ in (\ref{eq:berg-comp-bis}) would violate (\ref{eq:cond1}) and (\ref{eq:cond2}).
\end{example}

\paragraph{Norphism schemas}

So far, we have not discussed heuristics for actually choosing a set $\Nom$ for each pair of objects in $\Berg$. Here are some different ways.
\begin{compactenum}
    \item \textbf{Non-negativity of lenths}. Since path lengths cannot be negative, for all pair of objects $\Obja, \Objb$ we can say that we have a norphism
          \begin{equation}\label{eq:nor-schema-zero}
              \stylenorph{0}\colon \Obja \norto \Objb.
          \end{equation}
          If these are our only norphisms, we are providing no new information about paths.
    \item \textbf{Bound based on distance in $\reals^3$}.
          Any path along the mountain cannot be shorter than the distance of a straight line (``as the crow flies'').
          Therefore, for two objects $\tupp{\pos^1, \vel^1}$, $\tupp{\pos^2, \vel^2}$, we might
          choose the distance $\| \pos^1 - \pos^2 \| $ in $\reals^3$
          \begin{equation}\label{eq:nor-schema-distance-R3}
              \| \pos^1 - \pos^2 \| \colon \tupp{\pos^1, \vel^1} \norto \tupp{\pos^2, \vel^2}.
          \end{equation}
          as a norphism.
    \item \textbf{Bound based on geodesic distance}.
          More accurate bounds are given by taking geodesic distance as our norphisms.
          This is defined using the metric $d_\manifold$ of the manifold:
          \begin{equation}\label{eq:nor-schema-distance-geodesic}
              d_\manifold(\pos^1, \pos^2) \colon \tupp{\pos^1, \vel^1} \norto \tupp{\pos^2, \vel^2}.
          \end{equation}
    \item \textbf{Bound based on steepness interval}.
          A different kind of norphism is to encode steepness information, and relate it to the steepness of paths, instead of their length.
          Given two objects $\tupp{\pos^1, \vel^1}$, $\tupp{\pos^2, \vel^2}$, we can use one of the following bounds
          \begin{equation}\label{eq:nor-schema-altitudeU}
              \prfcomma{\pos^1_z - \pos^2_z<0}{
                  |\pos^1_z - \pos^2_z| /\steepnessU\colon \tupp{\pos^1, \vel^1} \norto \tupp{\pos^2, \vel^2}
              }
          \end{equation}
          \begin{equation}\label{eq:nor-schema-altitudeL}
              \prfperiod{\pos^1_z - \pos^2_z>0}{
                  |\pos^1_z - \pos^2_z| /\steepnessL\colon \tupp{\pos^1, \vel^1} \norto \tupp{\pos^2, \vel^2}
              }
          \end{equation}
\end{compactenum}

    \section{Example: co-design}
\label{sec:dp}

The next example revolves around the construction of norphisms for the category of design problems~\DP~\cite{censi2022,fong2018seven}; this is called $\textbf{Feas}_{\Bool}$ in~\cite{fong2018seven}.
The objects of \DP are posets.
The morphisms are design problems (also referred to as feasibility relations or boolean profunctors).
A \emph{design problem} (DP) $\adpa \colon \funposA \profto \resposB$ is a {monotone map} of the form $\adp \colon \funposA\op \Ptimes \resposB \toinPos \Bool,$ where~$\funposA, \resposB$ are arbitrary {posets} and $\Bool$ denotes the poset with elements $\makeset{\false, \true}$, with $\false \leq \true$.
The semantics for a DP is that it describes a process which provides a certain functionality, by requiring certain resources.
A design problem $\adp$ is a monotone map, since lowering the requested functionalities will not require more resources, and increasing the available resources will not provide less functionalities.

Morphism composition is defined as follows.
Given DPs $\adpa \colon \funposA \profto \resposB$ and~$\adpb \colon \funposB \profto \resposC$, their composite is
\begin{equation}
    \defmapperiod{
        (\adpa\functhen \adpb)
    }{
        \funposA\posop \Ptimes \resposC
    }{
        \toinPos
    }{
        \Bool
    }{
        \tup{\funposAopel, \resposCel}
    }{
        \bigvee_{\posBel \setin \funposB}
        \adpa(\funposAopel,\resposBel)
        \booland
        \adpb(\funposBopel,\resposCel)
    }
\end{equation}
For any poset~\posA, the {identity DP}~$\catidof\posA \colon \funposA \profto \resposA$ is the monotone map
\begin{equation}
    \label{eq:identity}
    \defmapperiod{
        \catid_\funposA
    }{
        \funposA\posop \Ptimes \resposA
    }{
        \toinPos
    }{
        \Bool
    }{
        \tupp{\funposAel_\F{1}\F{\opel},\resposAel_\R{2}}
    }{
        \funposAel_\F{1} \posleqof{\funposA} \resposAel_\R{2}
    }
\end{equation}

\paragraph*{Interpretation of norphisms in \DP}
Given that the morphisms of \DP are feasibility relations, we expect that the norphisms of \DP (``nesign problems'' NP), should be \emph{infeasibility} relations.
We define a nesign problem 
$\andpa\colon \funsp \nprofto \ressp$
to be a monotone map $\andpa\colon \funsp \Ptimes \ressp\op \to \Bool$, and we interpret $\andpa(f,r\opel)=\true$ to mean that it is \emph{not} possible to produce $f$ from $r$.
The idea is that if $\tup{f_1, r_1}$ is infeasible, then $f_1 \posleqof{} f_2$ implies that $\tup{f_2, r_1}$ is also infeasible and $r_2 \posleqof{} r_1$ implies that $\tup{f_1, r_2}$ is also infeasible. Note that the source poset of a nesign problem is the $\op$ of the source poset for a design problem.

\paragraph*{Compatibility of morphisms and norphisms}
Consider a DP $\adp\colon \funsp \profto \ressp$ and a NP $\andpa\colon \funsp \nprofto \ressp$.
The compatibility relation between DP and NP should ensure that there are no contradictions.
We ask that, for any pair of functionality/resources $\tup{f,r}$, it cannot happen that they are declared feasible by the DP ($\adpa(f\opel,r)$) \emph{and} declared infeasible by the NP ($\andpa(f,r\opel)$):
\begin{equation}
    \nmincompat{\funsp}{\ressp} (\andpa, \adp) = \exists f\setin \funsp, r\setin \ressp \colon \adpa(f\opel,r)\wedge \andpa(f,r\opel).
\end{equation}
\paragraph*{Composition rules for norphisms}
Given a NP $\andpa\colon \posA \nprofto \posB$ and a DP $\adpa\colon \posC \profto \posB$, one can compose them to get a NP $\northena{\andpa}{\adpa}\colon \posA \nprofto \posC$:
\begin{equation}
    (\northena{\andpa}{\adpa})(\posAel,\posCel\opel)=\bigvee_{\posBel \setin \posB}\andpa(\posAel, \posBel\opel) \booland \adp(\posCel\opel, \posBel).
\end{equation}
And given a DP $\adpa\colon \posB\profto \posA$ and a NP $\andpa\colon \posB \nprofto \posC$, one can compose them to get a NP $\adpa \northenbsymb \andpa\colon \posA \nprofto \posC$:
\begin{equation}
    (\northenb{\adpa}{\andpa})(\posAel,\posCel\opel)=\bigvee_{\posBel \setin \posB}\adp(\posBel\opel, \posAel)\booland \andpa(\posBel, \posCel\opel).
\end{equation}

The composition rules satisfy (\ref{eq:cond1}) and (\ref{eq:cond2}) and are exact, as may easily be checked.

\begin{example}
        \label{rem:example-ndp}
        Consider the posets $\posA=\tup{\natnumbers_{[\mathrm{kg \ pears}]},\leq}$, $\posB=\tup{\mathbb{R}_{\geq 0,[\mathrm{CHF}]},\leq }$, and $\posC=\tup{\natnumbers_{[\mathrm{kg \ raisins}]},\leq}$.
        Consider the design problem $\adp\colon \posC \profto \posB$ and the nesign problem $\andpa\colon \posA \nprofto \posB$ given, respectively, by the (in)feasibility relations
        \begin{equation*}
            \prfdoublecomma{\adpa(\F{r}\opel,\R{q})}{\F{r}\cdot 10 \leq \R{q}}\qquad
            \prfdoubleperiod{\andpa(\F{p},\R{q}\opel)}{\F{p}\cdot 5 > \R{q}}
        \end{equation*}
        These say that it is possible to buy raisins at \unit[10]{CHF/kg} or more, and never possible to buy pears at less than \unit[5]{CHF/kg}.
        We can evaluate the composition $(\northena{\andpa}{\adpa})\colon \posA \nprofto \posC$ in a particular point to understand its meaning. For instance:
        \begin{equation*}
            \begin{aligned}
                (\northena{\andpa}{\adpa})(\F{10},\R{4}\opel) & =\bigvee_{\posBel \setin \posB}\andpa(10, \posBel\opel) \booland \adp(4\opel, \posBel) \\
                                                              & =\bigvee_{\posBel \setin \posB} (40\leq q < 50)=\true.
            \end{aligned}
        \end{equation*}
        This equation is saying that we cannot get 10 kilos of pears from 4 kilos of raisins. The rationale is that, if I could, then I would be able to start with $40$ CHF and use $\adpa$ to get  $4$ kilos of raisins, which I could then use to obtain $10$ kilos of pears. But this would contradict the norphism $\andpa$, because $\andpa(10,40) = \true$ holds and this means that it is infeasible to exchange $40$ CHF for 10 kilos of pears. 
\end{example}

\paragraph*{Norphism schemas}
Considerations about how to define norphisms might follow from specific knowledge about particular designs that we know are (in)feasible, as well as from more general principles of physics or information theory. One very general rule that is arguably valid across all fields: in this universe, physically realizable designs can never produce strictly more of the same resource than one started with. This rule can be encoded as a norphism.
For each object $\posA$, we postulate a NP
\begin{equation}\label{eq:postulate-NP}
    \andpa_\posA \colon \posA\nprofto \posA,
\end{equation}
such that
\begin{equation}\label{eq:postulate-NP-val}
    \andpa_\posA(\posBel,\posAel\opel) = \posAel \posprecof{\posA} \posBel,
\end{equation}
where $\posAel \posprecof{\posA} \posBel = (\posAel \posleqof{\posA} \posBel ) \wedge (\posAel \neq \posBel )$.

Interestingly, starting from any morphism
\begin{equation}\label{eq:adp}
    \adp\colon \funsp \profto \ressp,
\end{equation}
one can directly obtain two NPs that go in the opposite direction, $\ressp \nprofto \funsp$.
These are
\begin{equation}\label{eq:inverse1}
    (\northena{\andpa_\ressp}{\adp})(r,f\opel)=\bigvee_{r'\setin \ressp}\andpa_\ressp(r,{r'}\opel)\wedge \adp(f\opel,r'),
\end{equation}
\begin{equation}\label{eq:inverse2}
    (\northenb{\adp}{\andpa_\funsp})(r,f\opel)=\bigvee_{f'\setin \funsp} \adp(f{'}\opel,r)\wedge \andpa_\funsp(f',f\opel).
\end{equation}
which gives two impossibility results.
The first states infeasibility because, while it is possible to get $f$ from $r'$ via $\adp$ for a certain $r'$, it is not possible to obtain $r$ from $r'$.
The second states infeasibility because, while it is possible to get $f'$ from $r$ via $\adp$ for a certain $f'$, it is not possible to obtain $f'$ from $f$.
In this nategory, we see that \emph{positive information induces negative information} in the other direction.

    %
    %

\section{The category $\GSet$}\label{sec:GSet}

The dialectica construction $\GC$ is due to De Paiva~\cite{de1989dialectica,de1989dialecticabis}, and its instantiation in the case $\CatC = \Set$ has been studied from a ``questions and answers'' perspective, for example in~\cite{blass1995questions}.
We will focus on \GSet, however our discussion is also interesting for other cases of the $\GC$ construction.
\proceedings{}{
    Recently, the $\GG(\CatC)$ construction and polynomial functors have been related to each other via a common generalization ~\cite{niu2022}.
}

\begin{definition}[\GSet]\label{def:GSet}
    An object of \GSet is a tuple
    \begin{equation}
        \tup{\gcQ, \gcA, \gcR},
    \end{equation}
    where $\gcQ$ and $\gcA$ are sets, and $\gcR : \gcQ \toinRel \gcA$ is a relation.

    A morphism
    $
        \gcma \colon \tupp{\gcQ_1, \gcA_1, \gcR_1}
        \mtoin\GC
        \tupp{\gcQ_2, \gcA_2, \gcR_2}
    $
    is a pair of maps
    \begin{align}
        \gcma = & \,\tupp{\gcmab, \gcmaf}, \\
                & \ \gcmab: \gcQ_1 \mfromin\Set \gcQ_2, \\
                & \ \gcmaf: \gcA_1 \mtoin\Set \gcA_2,
    \end{align}
    that satisfy the property
    \begin{equation}\label{eq:010-gset-mor-cond}
        \forall \gcq_2: \gcQ_2 \quad \forall \gca_1: \gcA_1 \quad
        \inrel{\gcmab(\gcq_2)}{\gcR_1}{\gca_1}
        \ \Imp\
        \inrel{\gcq_2}{\gcR_2}{\gcmaf(\gca_1)}.
    \end{equation}
    Morphism composition is defined component-wise
    \begin{align}
        \gcmB{\theseries} = \gcmbb \functhen \gcmab, \\
        \gcmF{\theseries} = \gcmaf \functhen \gcmbf,
    \end{align}
    and satisfies (\ref{eq:010-gset-mor-cond}) via composition of implications.

    The identity at $\tup{\gcQ, \gcA, \gcR}$ is
    $
        \catidat{\tup{\gcQ, \gcA, \gcR}} =
        \tup{
            \catidat{\gcQ}, \catidat{\gcA}
        }.
    $
\end{definition}

\begin{remark}
    Our notation was chosen to facilitate a ``questions and answers'' interpretation~\cite{blass1995questions}.
    In this perspective, an object of \GSet is a ``problem'': a relation $\gcR$ between
    a set of questions $\gcQ$ and a set of answers $\gcA$.
    For a particular question $\gcq\in\gcQ$
    and answer $\gca\in\gcA$, $\inrel{\gcq}{\gcR}{\gca}$ means that the answer is correct for the question.
    A morphism $
        \gcma \colon \tupp{\gcQ_1, \gcA_1, \gcR_1}
        \mtoin\GC
        \tupp{\gcQ_2, \gcA_2, \gcR_2}
    $ is a \emph{reduction} of problem 2 to problem 1, in the sense that we can use a solution to problem 1 to solve problem 2.
    The idea is that we start from a question $\gcq_2$ and transform it to a question $\gcq_1 = \gcmab(\gcq_2)$ of the first problem. Assuming we can find an answer $\gca_1$ to $\gcq_1$, we can then transform it in an answer of the second problem $\gca_2 = \gcmaf(\gca_1)$. The condition (\ref{eq:010-gset-mor-cond}) ensures that the answer so produced is correct for the second problem.
\end{remark}

We now rewrite the objects of \GSet in a slightly different way.
Instead of
\begin{equation}
    \gcR : \gcQ \toinRel \gcA,
\end{equation}
we can write this relation as a boolean function
\begin{equation}
    \label{eq:gcc-bool}
    \gcC \colon \gcQ \cartprod \gcA \to \makeset{\false, \true}.
\end{equation}
Letting $\Bool$ denote the category with two objects ``$\false$'' and ``$\true$'' and a single non-identity morphism $\Imp\colon \false \mto_{\Bool} \true$, we can rewrite (\ref{eq:gcc-bool}) again as
\begin{equation}\label{eq:dependent-implications}
    \gcC \colon \gcQ \cartprod \gcA \to \Obof \Bool.
\end{equation}
And then condition (\ref{eq:010-gset-mor-cond}) can be rewriten as a dependent function
\begin{equation}\label{eq:depfun-implication}
    \gcmae : \{\gcq_2: \gcQ_2, \gca_1: \gcA_1 \}
    \to \gcC_1 (\gcmab(\gcq_2), \gca_1)
    \mtoin\Bool \gcC_2 ( \gcq_2, \gcmaf(\gca_1) ).
\end{equation}
The value $\gcmae(\gcq_2, \gca_1)$ is a morphism in $\Bool$ that witnesses an implication.

\begin{remark}
    One idea that we find interesting is to replace $\Bool$ with some other category \CatB (on which we may wish to place suitable assumptions) and can consider maps of the form
    \begin{equation}
        \gcC \colon \gcQ \cartprod \gcA \to \Obof \CatB.
    \end{equation}
    If $\CatB$ is some category whose morphisms are ``proofs'', the analogue of (\ref{eq:depfun-implication})
    chooses a proof which is just one among many possible proofs (and such proofs might themselves be ordered by relevance or other criteria).
\end{remark}

\proceedings{}{
    \begin{definition}[Category $\GSetB$]\label{def:GSetB}
        Let \CatB be any category.
        An object of the category $\GSetB$ is a tuple
        \begin{equation}
            \tup{\gcQ, \gcA, \gcC},
        \end{equation}
        where $\gcQ$ is a set; $\gcA$ is a set, $\gcC$ is a function
        \begin{equation}
            \gcC \colon \gcQ \cartprod \gcA \to \Obof\CatB.
        \end{equation}
        A morphism
        $
            \gcma \colon \tupp{\gcQ_1, \gcA_1, \gcC_1}
            \mto%
            \tupp{\gcQ_2, \gcA_2, \gcC_2}
        $
        is a tuple of three functions
        \begin{align}
            \gcma = & \,\tupp{\gcmab, \gcmaf, \gcmae}, \\
                    & \ \gcmab: \gcQ_1 \mfromin\Set \gcQ_2, \\
                    & \ \gcmaf: \gcA_1 \mtoin\Set \gcA_2, \\
                    & \gcmae : \{\gcq_2: \gcQ_2, \gca_1: \gcA_1 \} \label{eq:depfun}
            \to \gcC_1 (\gcmab(\gcq_2), \gca_1)
            \mtoin\CatB \gcC_2 ( \gcq_2, \gcmaf(\gca_1) ).
        \end{align}
        The composition of the above morphism $\gcma$ with
        $
            \gcmb \colon \tupp{\gcQ_2, \gcA_2, \gcC_2}
            \mto %
            \tupp{\gcQ_3, \gcA_3, \gcC_3}
        $
        is defined as follows:
        \begin{align}
            \gcmB{\theseries} & = \gcmbb \functhen \gcmab, \\
            \gcmF{\theseries} & = \gcmaf \functhen \gcmbf, \\
            \gcmE{\theseries} & \colon
            \tupp{\gcq_3, \gca_1}
            \mapsto
            \gcmae (\gcmbb(\gcq_3) , \gca_1)
            \mthenof\CatB
            \gcmbe (\gcq_3 , \gcmaf(\gca_1) ).
        \end{align}
        More explicitly,
        \begin{equation}
            \begin{aligned}
                \gcmE{\theseries}
                 &
                \colon
                \tupp{\gcq_3, \gca_1}
                \mapsto \\
                 & \gcC_1( (\gcmbb \functhen \gcmab)(\gcq_3), \gca_1)
                \xrightarrow{\mbox{$\gcmae (\gcmbb(\gcq_3) , \gca_1 )$}}
                \gcC_2( \gcmbb(\gcq_3), \gcmaf(\gca_1))
                \xrightarrow{\mbox{$\gcmbe (\gcq_3 , \gcmaf(\gca_1) )$}}%
                \gcC_3(\gcq_3, (\gcmaf \functhen \gcmbf) (\gca_1)).
            \end{aligned}
        \end{equation}
        The identity at $\tupp{\gcQ, \gcA, \gcC}$ is given by $\tupp{\catidat{\gcQ}, \catidat{\gcA},
                \tupp{\gcq, \gca} \mapsto \catidat{\gcC(\gcq,\gca)}}$.
    \end{definition}

}
\proceedings{}{
    For completeness, we give the further generalization, in which we have \Category instead of \Set and $\gcmae$ becomes a natural transformation.

    \begin{definition}[$\GG(\Category, \gcE)$]\label{def:GCatB}
        Given a category $\gcE$,
        an object of $\GG(\Category, \gcE)$ is a tuple
        \begin{equation}
            \tup{\gcQ, \gcA, \gcC },
        \end{equation}
        where $\gcQ$ is a category, $\gcA$ is a category, $\gcC$ is a functor
        \begin{equation}
            \gcC \colon \gcQ \op \times \gcA \to \gcE.
        \end{equation}
        A morphism
        $
            \gcma \colon \tupp{\gcQ_1, \gcA_1, \gcC_1}
            \mtoin\GCat
            \tupp{\gcQ_2, \gcA_2, \gcC_2}
        $
        is a tuple
        \begin{equation}
            \gcma = \tupp{\gcmab, \gcmaf, \gcmae},
        \end{equation}
        where
        \begin{itemize}
            \item $\gcmab: \gcQ_2 \mtoin{\Category} \gcQ_1$ is a functor,
            \item $\gcmaf: \gcA_1 \mtoin{\Category} \gcA_2$ is a functor,
            \item $\gcmae$ is a natural transformation between two functors
                  \begin{equation}
                      \funa, \funb: \gcQ_2\op \times \gcA_1 \to \gcE,
                  \end{equation}
                  defined as
                  \begin{align}
                      \funa & = (\gcmab \times \funidat{\gcA_1}) \fthen \gcC_1, \\
                      \funb & = (\funidat{\gcQ_2\op} \times \gcmaf) \fthen \gcC_2 .
                  \end{align}
        \end{itemize}
        \todo[inline]{Need to finish with composition}
    \end{definition}
}

\proceedings{}{
    \begin{table}[!h]
        \caption{Linear logic binary-operators}
        \begin{tabular}{cccc}
                            & connective & unit      & \\
            multiplicatives & $\lltimes$ & $\llone$  & \\
                            & $\llpar$   & $\llbot$  & \\
            additives       & $\llplus$  & $\llzero$ & \\
                            & $\llwith$  & $\lltop$
        \end{tabular}
    \end{table}
}

\proceedings{}{
    \subsection{Monoidal product $\gcMeh$}
    \begin{definition}[Monoidal product $\gcMeh$]\label{def:meh}
        Assume \CatB has finite products.
        The action on the objects is defined as follows:
        \begin{equation}
            \tupp{\gcQ_1, \gcA_1, \gcC_1}
            \gcMehob
            \tupp{\gcQ_2, \gcA_2, \gcC_2}
            = \tupp{
                \gcQ_1 \cartprod \gcQ_2,
                \gcA_1 \cartprod \gcA_2,
                \gcC_1 \gcCmeh \gcC_2
            }
        \end{equation}
        \begin{equation}
            \gcC_1 \gcCmeh \gcC_2 : \tupp{\tupp{\gcq_1, \gcq_2}, \tupp{\gca_1, \gca_2}}
            \mapsto \gcC_1(\gcq_1, \gca_1) \cprodob\CatB \gcC_2(\gcq_2, \gca_2 ),
        \end{equation}
        where $\cprodob\CatB$ is the product of two objects in $\CatB$.
        The monoidal unit is
        \begin{equation}
            1_{\gcMeh} = \tupp{\makeset{\theonlyQ}, \makeset{\theonlyA}, \gctrue},
            \qquad \gctrue \colon \tupp{\theonlyQ, \theonlyA} \mapsto \termobof\CatB.
        \end{equation}
        The product of
        $
            \gcma \colon \tupp{\gcQ_1, \gcA_1, \gcC_1}
            \mto
            \tupp{\gcQ_3, \gcA_3, \gcC_3} $ and $
            \gcmb \colon \tupp{\gcQ_2, \gcA_2, \gcC_2}
            \mto
            \tupp{\gcQ_4, \gcA_4, \gcC_4}
        $ is
        \begin{equation}
            \gcma \gcMehmor \gcmb \colon
            \tupp{
                \gcQ_1 \cartprod \gcQ_2 ,
                \gcA_1 \cartprod \gcA_2,
                \gcC_1 \gcCmeh \gcC_2
            }
            \mto
            \tupp{
                \gcQ_3 \cartprod \gcQ_4,
                \gcA_3 \cartprod \gcA_4,
                \gcC_3 \gcCmeh \gcC_4
            }
        \end{equation}
        \begin{align}
            \gcmB{\themeh} & = \gcmab \funcprod \gcmbb, \\
            \gcmF{\themeh} & = \gcmaf \funcprod \gcmbf, \\
            \gcmE{\themeh} & \colon \tupp{\tupp{\gcq_3, \gcq_4}, \tupp{\gca_1, \gca_2}}
            \mapsto
            \gcmae ( \gcq_3, \gca_1)
            \cprodmor\CatB
            \gcmbe (\gcq_4, \gca_2).
            \label{eq:meh-m}
        \end{align}
    \end{definition}

    \subsection{Monoidal product $\lltimes$}
    \begin{definition}[Monoidal product $\lltimes$]\label{def:lltimes}
        Assume \CatB has finite products.
        The action on the objects is defined as follows:
        \begin{equation}
            \tupp{\gcQ_1, \gcA_1, \gcC_1}
            \lltimesob
            \tupp{\gcQ_2, \gcA_2, \gcC_2}
            = \tupp{
                \gcQ_1 ^ {\gcA_2} \cartprod \gcQ_2 ^ {\gcA_1},
                \gcA_1 \cartprod \gcA_2,
                \gcC_1 \gcCtimes \gcC_2
            }
        \end{equation}
        \begin{equation}
            \gcC_1 \gcCtimes \gcC_2 : \tupp{\tupp{\gcq_1, \gcq_2}, \tupp{\gca_1, \gca_2}}
            \mapsto \gcC_1(\gcq_1(\gca_2), \gca_1) \cprodob\CatB \gcC_2(\gcq_2 (\gca_1), \gca_2 ),
        \end{equation}
        where $\cprodob\CatB$ is the product of two objects in $\CatB$.
        The monoidal unit is
        \begin{equation}
            1_{\otimes} = \tupp{\makeset{\theonlyQ}, \makeset{\theonlyA}, \gctrue},
            \qquad \gctrue \colon \tupp{\theonlyQ, \theonlyA} \mapsto \termobof\CatB.
        \end{equation}
        The product of
        $
            \gcma \colon \tupp{\gcQ_1, \gcA_1, \gcC_1}
            \mto
            \tupp{\gcQ_3, \gcA_3, \gcC_3} $ and $
            \gcmb \colon \tupp{\gcQ_2, \gcA_2, \gcC_2}
            \mto
            \tupp{\gcQ_4, \gcA_4, \gcC_4}
        $ is
        \begin{equation}
            \gcma \lltimesmor \gcmb \colon
            \tupp{
                \gcQ_1 ^ {\gcA_2} \cartprod \gcQ_2 ^ {\gcA_1},
                \gcA_1 \cartprod \gcA_2,
                \gcC_1 \gcCtimes \gcC_2
            }
            \mto
            \tupp{
                \gcQ_3 ^ {\gcA_4} \cartprod \gcQ_4 ^ {\gcA_3},
                \gcA_3 \cartprod \gcA_4,
                \gcC_3 \gcCtimes \gcC_4
            }
        \end{equation}
        \begin{align}
            \gcmB{\theprod} & =
            \tupp{
                \gcmbf \functhen - \functhen \gcmab,
                \gcmaf \functhen - \functhen \gcmbb
            }, \\
            \gcmF{\theprod} & = \gcmaf \funcprod \gcmbf, \\
            \gcmE{\theprod} & \colon \tupp{\tupp{\gcq_3, \gcq_4}, \tupp{\gca_1, \gca_2}}
            \mapsto
            \gcmae ( (\gcmbf\functhen\gcq_3)(\gca_2), \gca_1)
            \cprodmor\CatB
            \gcmbe ( (\gcmaf\functhen\gcq_3)(\gca_1), \gca_2).
            \label{eq:otimes-m}
        \end{align}
    \end{definition}
    The ``$\cprodmor\CatB$'' in (\ref{eq:otimes-m}) is the product of the two morphisms in $\CatB$.

    \subsubsection{Derivation of (\ref{eq:otimes-m})}
    The function $\gcmE{\theprod}$ needs to have type
    \begin{align}
        \gcmE{\theprod} : & \{\tupp{\gcq_3, \gcq_4} : \gcQ_3 ^ {\gcA_4} \cartprod \gcQ_4 ^ {\gcA_3},
        \tupp{\gca_1, \gca_2}: \gcA_1 \cartprod \gcA_2 \}
        \to \\
                          & \gcC_1 \gcCtimes \gcC_2 \pars{\gcmB{\theprod}(\tupp{\gcq_3, \gcq_4}) , \tupp{\gca_1, \gca_2}}
        \mtoin\CatB \\
                          & \gcC_3 \gcCtimes \gcC_4 \pars {\tupp{\gcq_3, \gcq_4}, \gcmaf \funcprod \gcmbf(\tupp{\gca_1, \gca_2}) }
    \end{align}
    Simplifying, we get that we should find a map
    \begin{align}
        \gcmE{\theprod} & \colon \tupp{\tupp{\gcq_3, \gcq_4}, \tupp{\gca_1, \gca_2}}
        \mapsto \\
                        & \gcC_1 \gcCtimes \gcC_2 \pars{\tupp{\gcmbf \functhen \gcq_3 \functhen \gcmab, \gcmaf \functhen \gcq_4 \functhen \gcmbb} , \tupp{\gca_1, \gca_2}}
        \mtoin\CatB \\
                        & \gcC_3 \gcCtimes \gcC_4 \pars {\tupp{\gcq_3, \gcq_4}, \tupp{\gcmaf(\gca_1), \gcmbf(\gca_2)} }
    \end{align}
    Expanding $\gcC_1 \gcCtimes \gcC_2$ and $\gcC_3 \gcCtimes \gcC_4$ we have
    \begin{align}
        \gcmE{\theprod} & \colon \tupp{\tupp{\gcq_3, \gcq_4}, \tupp{\gca_1, \gca_2}}
        \mapsto \\
                        & \gcC_1(\gcmab(\gcmbf \functhen \gcq_3 (\gca_2)), \gca_1) \cprodob\CatB
        \gcC_2(\gcmbb(\gcmaf \functhen \gcq_4 (\gca_1)), \gca_2 )
        \mtoin\CatB \\
                        & \gcC_3(\gcmbf \functhen \gcq_3(\gca_2), \gcmaf(\gca_1))
        \cprodob\CatB
        \gcC_4(\gcmaf \functhen \gcq_4(\gca_1), \gcmbf(\gca_2) )
    \end{align}
    Now notice that we can use $\gcmae$ and $\gcmbe$ to obtain the two morphisms
    \begin{align}
        \gcmae ( (\gcmbf\functhen\gcq_3)(\gca_2), \gca_1): \gcC_1(\gcmab(\gcmbf \functhen \gcq_3 (\gca_2)), \gca_1)
        \mtoin\CatB \gcC_3(\gcmbf \functhen \gcq_3(\gca_2), \gcmaf(\gca_1))
        \\
        \gcmbe ( (\gcmaf\functhen\gcq_3)(\gca_1), \gca_2): \gcC_2(\gcmbb(\gcmaf \functhen \gcq_4 (\gca_1)), \gca_2 )
        \mtoin\CatB \gcC_4(\gcmaf \functhen \gcq_4(\gca_1), \gcmbf(\gca_2) )
    \end{align}

    Therefore, we can take the product of the two morphisms and obtain a morphism of the right type:
    \begin{equation}
        \gcmE{\theprod} \colon \tupp{\tupp{\gcq_3, \gcq_4}, \tupp{\gca_1, \gca_2}}
        \mapsto
        \gcmae ( (\gcmbf\functhen\gcq_3)(\gca_2), \gca_1)
        \funcprod
        \gcmbe ( (\gcmaf\functhen\gcq_3)(\gca_1), \gca_2).
    \end{equation}

    \subsection{Monoidal product $\llpar$}
    \begin{definition}[Monoidal product $\llpar$]\label{def:llpar}
        Assume \CatB has finite coproducts.
        The action on the objects is defined as follows:
        \begin{equation}
            \tupp{\gcQ_1, \gcA_1, \gcC_1}
            \llparob
            \tupp{\gcQ_2, \gcA_2, \gcC_2}
            = \tupp{
                \gcQ_1 \cartprod \gcQ_2,
                \gcA_1 ^ {\gcQ_2} \cartprod \gcA_2 ^ {\gcQ_1},
                \gcC_1 \gcCpar \gcC_2
            }
        \end{equation}
        \begin{equation}
            \gcC_1 \gcCpar \gcC_2 : \tupp{\tupp{\gcq_1, \gcq_2}, \tupp{\gca_1, \gca_2}}
            \mapsto \gcC_1(\gcq_1(\gca_2), \gca_1) \csumob\CatB \gcC_2(\gcq_2(\gca_1), \gca_2 ),
        \end{equation}
        where $\csumob\CatB$ is the coproduct of two objects in $\CatB$.
        The monoidal unit is
        \begin{equation}
            1_{\llpar} = \tupp{\makeset{\theonlyQ}, \makeset{\theonlyA}, \gcfalse}, \qquad \gcfalse \colon \tupp{\theonlyQ, \theonlyA} \mapsto \initobof\CatB.
        \end{equation}
        The product of
        $
            \gcma \colon \tupp{\gcQ_1, \gcA_1, \gcC_1}
            \mto
            \tupp{\gcQ_3, \gcA_3, \gcC_3}$, $
            \gcmb \colon \tupp{\gcQ_2, \gcA_2, \gcC_2}
            \mto
            \tupp{\gcQ_4, \gcA_4, \gcC_4}
        $
        is
        \begin{equation}
            \gcma \llparmor \gcmb \colon
            \tupp{
                \gcQ_1 \cartprod \gcQ_2,
                \gcA_1 ^ {\gcQ_2} \cartprod \gcA_2 ^ {\gcQ_1},
                \gcC_1 \gcCpar \gcC_2
            }
            \mto
            \tupp{
                \gcQ_3 \cartprod \gcQ_4,
                \gcA_3 ^ {\gcQ_4} \cartprod \gcA_4 ^ {\gcQ_3},
                \gcC_3 \gcCpar \gcC_4
            }
        \end{equation}
        \begin{align}
            \gcmB{\thepar} & = \gcmab \funcprod \gcmbb, \\
            \gcmF{\thepar} & =
            \tupp{
                \gcmbb \functhen - \functhen \gcmaf,
                \gcmab \functhen - \functhen \gcmbf
            }, \\
            \gcmE{\thepar} & \colon \tupp{\tupp{\gcq_3, \gcq_4}, \tupp{\gca_1, \gca_2}}
            \mapsto
            \gcmae ( \dots, \gca_1)
            \csummor\CatB
            \gcmbe (\dots, \gca_2).
            \label{eq:par-m}
        \end{align}
        where $\csummor\CatB$ is the coproduct of two morphisms in $\CatB$.
    \end{definition}
    \todo[inline]{Finish formula above}
}

\subsection{A monoidal product for $\GSet$ }

The categories $\GC$ have a very rich structure.
In particular they provide models of linear logic with four distinct monoidal products $\lltimes$, $\llpar$, $\llplus$, and $\llwith$.
We define here a monoidal product $\gcPlus$ for $\GSet$ which one might say is ``in between'' the the multiplicative connectives $\lltimes$ and $\llpar$ (which are denoted $\oslash$ and $\square$, respectively, in \cite{de1989dialectica,de1989dialecticabis}).

\begin{definition}[Monoidal product $\gcPlus$]\label{def:gcPlus}
    On objects, 
    \begin{equation}
        \tupp{\gcQ_1, \gcA_1, \gcC_1}
        \gcPlusob
        \tupp{\gcQ_2, \gcA_2, \gcC_2}
        = \tupp{
            \gcQ_1 ^ {\gcA_2} \cartprod \gcQ_2 ^ {\gcA_1},
            \gcA_1 \cartprod \gcA_2,
            \gcC_1 \gcCplus \gcC_2
        },
    \end{equation}
    where
    \begin{equation}
        \gcC_1 \gcCplus \gcC_2: \tupp{\tupp{\gcq_1, \gcq_2}, \tupp{\gca_1, \gca_2}}
        \mapsto \gcC_1(\gcq_1(\gca_2), \gca_1) \vee \gcC_2(\gcq_2(\gca_1), \gca_2 ).
    \end{equation}

    The product of morphisms
    $
        \gcma \colon \tupp{\gcQ_1, \gcA_1, \gcC_1}
        \mto
        \tupp{\gcQ_3, \gcA_3, \gcC_3}$ and $
        \gcmb \colon \tupp{\gcQ_2, \gcA_2, \gcC_2}
        \mto
        \tupp{\gcQ_4, \gcA_4, \gcC_4}
    $ is
    \begin{equation}
        \gcma \gcPlusmor \gcmb \colon
        \tupp{
            \gcQ_1 ^ {\gcA_2} \cartprod \gcQ_2 ^ {\gcA_1},
            \gcA_1 \cartprod \gcA_2,
            \gcC_1 \gcCplus \gcC_2
        }
        \mto
        \tupp{
            \gcQ_3 ^ {\gcA_4} \cartprod \gcQ_4 ^ {\gcA_3},
            \gcA_3 \cartprod \gcA_4,
            \gcC_3 \gcCplus \gcC_4
        },
    \end{equation}
    with
    \begin{align}
        \gcmB{\thesum} & = %
        \tupp{
            \gcmbf \functhen - \functhen \gcmab,
            \gcmaf \functhen - \functhen \gcmbb
        }, \\
        \gcmF{\thesum} & = \gcmaf \funcprod \gcmbf, \\
        \gcmE{\thesum} & \colon \tupp{\tupp{\gcq_3, \gcq_4}, \tupp{\gca_1, \gca_2}}
        \mapsto
        \gcmae ( \gcmbf\functhen\gcq_3(\gca_2), \gca_1)
        \vee
        \gcmbe ( \gcmaf\functhen\gcq_3(\gca_1), \gca_2).
        \label{eq:oplus-sum}
    \end{align}
    The monoidal unit is
    \begin{equation}
        1_{\gcPlus} = \tupp{\makeset{\theonlyQ}, \makeset{\theonlyA}, \gcfalse},
        \qquad
        \gcfalse \colon \tupp{\theonlyQ, \theonlyA} \mapsto \false.
    \end{equation}
    For the associator and unitors we make the canonical choices, which are easily inferred from their signatures.
    We refrain from writing them out explicitly here.
    For reasons of space we also omit the the proof that $\tupp{\GSet,\gcPlus}$ is indeed a monoidal category.
\end{definition}

\begin{remark}
    In the generalization where we replace $\Bool$ with some category $\CatB$, the operation $\vee$ and the object $\false$ in $\Bool$ which are used in the above definition would be replaced by suitable substitutes in $\CatB$.
\end{remark}

    \ifthenelse{\boolean{proceedings}}{}{\clearpage}
    \section{Describing nategories using enrichment}
\label{sec:payoff}

We recall the following standard definition of enriched category \cite{kelly1982basic}, for easy reference and to fix notation.

\begin{definition}[Enriched category]
    \label{def:enrichment-structure}
    Let~$\tup{\CatV, \mtimescat, \idmoncat, \associator, \leftunitor,\rightunitor}$ be a {monoidal category}.
    \\
    A \CatV-enriched category \CatE is a tuple
    $
        \tupp{\Ob_\CatE, \homobjectset{}, \enm, \enid_{}},
    $
    where
    \begin{compactenum}
        \item $\Ob_\CatE$ is a collection of objects.
        \item $\homobjectset{}$ is a function such that, for all pairs of objects~$\Obja,\Objb\setin \Ob_\CatE$, its value~$\homobject{}{\Obja}{\Objb}$ is an object of~$\CatV$, called a \emph{hom-object}. %
        \item $\enm$ is a function such that, for all~$\Obja, \Objb, \Objc\setin \Ob_\CatE$, there exists a morphism~$\enmof{}{\Obja}{\Objb}{\Objc}$ of~\CatV
              \begin{equation}
                  \enmof{}{\Obja}{\Objb}{\Objc}\colon\homobject{}{\Obja}{\Objb} \mtimescatob \homobject{}{\Objb}{\Objc} \mtoin\CatV \homobject{}{\Obja}{\Objc},
              \end{equation}
              called a \emph{composition morphism}.

        \item $\enid_{}$ is a function such that, for each~$\Obja\setin \ObE$, there exists a morphism of~\CatV
              \begin{equation}
                  \enidof{}{\Obja} \colon \idmoncat \mtoin{\CatV} \homobject{}{\Obja}{\Obja},
              \end{equation}
              called an \emph{identity-choosing morphism}.
    \end{compactenum}
    Moreover, for any~$\Obja,\Objb,\Objc,\Objd\setin \ObE$, the following diagrams must commute.

    \begin{equation}
        \label{eq:enrichment-assoc2-act}
        \includesag{enrichment_assoc}
    \end{equation}
    \begin{equation}
        \label{eq:enrichment-unital2}
        \includesag{enrichment_unital}
    \end{equation}
\end{definition}
Recall that specifying the data of an ordinary (locally small) category is equivalent to specifying a category enriched in the monoidal category $\prcatp=\tup{\Set, \cartprod, \singleton}$.
In this case, both the enriched category and the ordinary category have the same objects, the hom-objects of the enriched category correspond to the hom-sets of the ordinary category, the composition morphisms encode the composition operations, and the identity-choosing morphisms select an element of each of the hom-sets of the type $\homobject{}{\Obja}{\Obja}$, corresponding to identity morphisms.
The diagrams (\ref{eq:enrichment-assoc2-act}) and (\ref{eq:enrichment-unital2}) encode the associativity of the composition operations, and that the identity morphisms act neutrally for composition.

\proceedings{}{
    \begin{proof}
        We show one direction.
        Suppose that we are given a $\prcatp$-enriched category as a tuple
        $\tupp{\Ob, \homobjectset{}, \enm, \enid_{}}.
        $
        We can define a small category $\CatC$  as follows:
        \begin{compactitem}
            \item Set $\ObC \definedas \ObE$.
            \item For each $\Obja,\Objb\setin\ObC$, let $\HomSet{\CatC}{\Obja}{\Objb} \definedas \homobject{}{\Obja}{\Objb}$.
            \item For each $\Obja,\Objb,\Objc \setin \ObC$, we know a function
                  \begin{equation}
                      \enmof{}{\Obja}{\Objb}{\Objc}:\HomSet{\CatC}{\Obja}{\Objb}\mtimescat \HomSet{\CatC}{\Objb}{\Objc}\mtoin{\Set} \HomSet{\CatC}{\Obja}{\Objc}.
                  \end{equation}

                  The diagrams constraints imply that this function is associative.
                  \\Therefore, we use it to define morphism composition in $\CatC$, setting $\mthen_{\Obja,\Objb,\Objc}\definedas \enmof{}{\Obja}{\Objb}{\Objc}$.
            \item For each $\Obja\setin\ObC$ we know a function $\enidof{}{\Obja}\colon\singleton\mtoin{\Set} \HomSet{\CatC}{\Obja}{\Obja}$ that selects a morphism.\\
                  The diagrams constraints imply that such morphism satisfies unitality
                  with respect to~$\mthen_{\Obja,\Objb,\Objc}$. \\Therefore, we can use it to define the identity at each object:
                  \begin{equation}
                      \label{eq:ja-hom}
                      \catidat{\Obja} \definedas \enidof{}{\Obja}(\singletonel).
                  \end{equation}
        \end{compactitem}
    \end{proof}
    The commutative diagrams are clearly satisfied, by choosing the obvious associator and left/right unitors.
}

The proof of our main result below follows a similar pattern -- we show that to specify a nategory which satisfies some additional conditions it is sufficient to specify a category enriched in the monoidal category $\tup{\GSet, \gcPlus}$.
We will denote $\tup{\GSet, \gcPlus}$ by \prcat, which stands for ``positive'' and ``negative''.

\begin{proposition}\label{prop:PN-enrich}
    A \prcat-enriched category provides the data necessary to specify a nategory.
    However, not all nategories can be specified by the data of a \prcat-enriched category,
    because the nategory produced has the following additional properties, which encode a covariant and a contravariant ``action'' of morphisms on norphisms.

    \noindent \emph{Identities act neutrally:}
    \begin{align}
        \northenb{\catidat{}}{\nora}  = \nora \tag{neut-1} \label{eq:neutrality1}, \\
        \northena{\nora}{\catidat{}} = \nora \tag{neut-2}\label{eq:neutrality2}.
    \end{align}

    \noindent \emph{Compatibility with composition}:
    \begin{align}
        \northenb{(\morab)}{\nora}            & =\northenb{\morb}{  (\northenb{\mora}{\nora})}, \tag{covar}\label{eq:compat-comp1} \\
        \northena{\nora}{ (\morb\mthen\morc)} & =\northena{(\northena{\nora}{\morc})}{\morb}. \tag{contravar}\label{eq:compat-comp2}
    \end{align}
    \noindent \emph{The actions commute:}
    \begin{align}
        \northenb{\mora}{ (\northena{\nora}{\morc})} & =\northena{(\northenb{\mora}{\nora}) }{ \morc}.\tag{comm}
    \end{align}
    These conditions are not satisfied by all nategories.
\end{proposition}

\begin{proof}
    Suppose somebody has provided us with a \prcat-enriched category $\CatE = \tupp{\ObE, \homobjectset{}, \enm, \enid_{}}$.
    Using this data we will describe a nategory $\CatC$ with the above-stated properties.

    For the objects of \CatC, we set $\ObC \definedas \ObE$.

    For every pair of objects $\Obja,\Objb \setin \ObC$, we have an object $\homobject{}{\Obja}{\Objb}$ of \prcat.
    This is a tuple
    \begin{equation}\label{eq:tup}
        \homobject{}{\Obja}{\Objb} = \tupp{\gcQ, \gcA, \gcC},
    \end{equation}
    which we interpret as
    \begin{equation}\label{eq:tup-interpret}
        \homobject{}{\Obja}{\Objb} =\tupp{\NomSet{\CatC}{\Obja}{\Objb}, \HomSet{\CatC}{\Obja}{\Objb}, \nmincompat{\Obja}{\Objb}},
    \end{equation}
    thereby setting $\NomSet{\CatC}{\Obja}{\Objb} \definedas \gcQ$, $\HomSet{\CatC}{\Obja}{\Objb}\definedas \gcA$, and $\nmincompat{\Obja}{\Objb} \definedas \gcC$.

    Next, for each $\Obja \setin \Ob$ we have an identity-choosing morphism
    \begin{equation}\label{eq:provided}
        \enidof{}{\Obja} \colon \idmoncat_\prcat \mtoin{\prcat} \homobject{}{\Obja}{\Obja}.
    \end{equation}
    Because $\idmoncat_\prcat=\tupp{\makeset{\theonlyQ}, \makeset{\theonlyA}, \gcfalse}$, this is a morphism
    \begin{equation}\label{eq:provided-unroll}
        \enidof{}{\Obja}
        \colon
        \tupp{\makeset{\theonlyQ}, \makeset{\theonlyA}, \gcfalse}
        \mtoin{\prcat}
        \tupp{
            \NomSet{\CatC}{\Obja}{\Obja},
            \HomSet{\CatC}{\Obja}{\Obja},
            \nmincompat{\Obja}{\Obja}
        },
    \end{equation}
    which consists of three functions $\tupp{\gcmab, \gcmaf ,\gcmae}$.
    The forward map $\gcmaf: \makeset{\theonlyA} \mto \HomSet{\CatC}{\Obja}{\Obja}$ chooses our (candidate) identity morphism, so we set $ \idname_\Obja \definedas \gcmaf(\theonlyA)$.
    The backward map $\gcmab
        \colon
        \NomSet{\CatC}{\Obja}{\Obja}
        \to
        \makeset{\theonlyQ}$
    is uniquely determined and does not carry any information.
    As for $\gcmae$, it is a dependent function of the type
    \begin{equation}
        \gcmae : \{\gcq_2: \NomSet{\CatC}{\Obja}{\Obja}, \gca_1: \makeset{\theonlyA} \}
        \to \gcfalse (\gcmab(\gcq_2), \gca_1)
        \mtoin\Bool \nmincompat{\Obja}{\Obja} ( \gcq_2, \gcmaf(\gca_1) ).
    \end{equation}
    Evaluated at $\gcq_2 = \nora$ and $\gca_1 = \theonlyA$, we have
    \begin{equation}
        \gcmae(\nora,\theonlyA) : \false \mtoin\Bool
        \nmincompat{\Obja}{\Obja} ( \nora, \catid_{\Obja} ).
    \end{equation}
    Because $\false$ is an initial object in $\Bool$, such a morphism always exists,
    no matter what the right-hand side is.
    Therefore, this condition does not carry any additional information.

    Now let us fix three objects $\Obja,\Objb,\Objc$ and consider the composition morphism
    \begin{equation}\label{eq:beta}
        \enmof{}{\Obja}{\Objb}{\Objc}
        \colon \homobject{}{\Obja}{\Objb}
        \mtimescatob_\prcat \homobject{}{\Objb}{\Objc}
        \mtoin\prcat \homobject{}{\Obja}{\Objc}.
    \end{equation}
    Rewriting the hom-objects as tuples and using abbreviated notation we have
    \begin{equation}\label{eq:beta-ours2}
        \enmof{}{\Obja}{\Objb}{\Objc} \colon \tupp{
            \Nab,
            \Hab,
            \nmincompat{\Obja}{\Objb}
        }
        \mtimescat_\prcat
        \tupp{
            \Nbc,
            \Hbc,
            \nmincompat{\Objb}{\Objc}
        }
        \mtoin\prcat \\
        \tupp{
            \Nac,
            \Hac,
            \nmincompat{\Obja}{\Objc}
        }.
    \end{equation}
    Expanding using the definition of $\mtimescat_\prcat$ we find
    \begin{equation}\label{eq:themorph}
        \enmof{}{\Obja}{\Objb}{\Objc} \colon
        \tupp{
            \Nab^{\Hbc} \cartprod \Nbc^{\Hab},
            \Hab \cartprod \Hbc,
            \nmincompat{\Obja}{\Objb} \gcCplus \nmincompat{\Objb}{\Objc}
        }
        \mtoin\prcat
        \tupp{
            \Nac,
            \Hac,
            \nmincompat{\Obja}{\Objc}
        }.
    \end{equation}
    Such a morphism corresponds to three maps $\tupp{\gcmbb, \gcmbf ,\gcmbe}$.
    The forward map $\gcmbf $ has type
    \begin{equation}\label{eq:morcomp}
        \gcmbf \colon \HomSet{\CatC}{\Obja}{\Objb}\cartprod \HomSet{\CatC}{\Objb}{\Objc}
        \to \HomSet{\CatC}{\Obja}{\Objc},
    \end{equation}
    and we use it to define morphism composition `` $\mthenof\CatC$'' in our nategory.
    The backward map $\gcmbb$ has type
    \begin{equation}\label{eq:con-back}
        \gcmbb \colon \Nac \to \Nab^{\Hbc} \cartprod \Nbc^{\Hab},
    \end{equation}
    which, after splitting into two maps (using the universal property of the product) and currying, specifies two maps
    \begin{align} \label{eq:con-back2}
        \northenasymb & \colon \Nac \cartprod \Hbc \to \Nab,
        \\
        \northenbsymb & \colon \Hab \cartprod \Nac \to \Nbc.
    \end{align}
    As for the dependent function $\gcmbe$, given $\nora: \Nac$, $\mora: \Hab$, $\morb: \Hbc$ we have
    \begin{equation}\label{eq:proc1}
        \gcmbe(\nora,\tup{\mora,\morb})
        \colon
        (\nmincompat{\Obja}{\Objb} \gcCplus \nmincompat{\Objb}{\Objc}) ( \tupp{ (\nora \northenasymb -), (- \northenbsymb \nora)}, \tup{\mora,\morb})
        \mtoin\Bool \nmincompat{\Obja}{\Objc} ( \nora, \morab ).
    \end{equation}
    Expanding more,
    \begin{equation}\label{eq:proc2}
        \gcmbe(\nora,\tup{\mora,\morb})
        \colon
        \nmincompat{\Obja}{\Objb} (\nora \northenasymb \morb, \mora)
        \vee
        \nmincompat{\Objb}{\Objc} (\mora \northenbsymb \nora, \morb)
        \mtoin\Bool
        \nmincompat{\Obja}{\Objc} ( \nora, \morab ),
    \end{equation}
    which is equivalent to having two maps
    \begin{align}
        \gcmbe_1(\nora,\tup{\mora,\morb})\label{eq:proc3-maps}
        \colon
        \nmincompat{\Obja}{\Objb} (\northena{\nora}{\morb} , \mora)
        \mtoin\Bool
        \nmincompat{\Obja}{\Objc} ( \nora, \morab ), \\
        \gcmbe_2(\nora,\tup{\mora,\morb})\label{eq:proc4-maps}
        \colon
        \nmincompat{\Objb}{\Objc}(\northenb{\mora}{\nora}, \morb)
        \mtoin\Bool
        \nmincompat{\Obja}{\Objc} ( \nora, \morab ).
    \end{align}
    These witness the implications which give us (\ref{eq:cond1}) and (\ref{eq:cond2}).

    Next we move to the commutative diagrams in the definition of enriched category.
    As a general observation, we note that for diagrams in $\GSet$ we only need to consider commutativity on the level of ``forward maps'' and ``backward maps'' respectively.
    We do not need to worry about the conditions (\ref{eq:010-gset-mor-cond}), because for any two parallel morphisms this condition is the same, and hence ``commutativity'' is trivially satisfied.

    \paragraph*{Conditions from the associativity diagram for enriched categories}
    We now consider the diagram (\ref{eq:enrichment-assoc2-act}), in the case of \prcat.
    \proceedings{}{
        \begin{center}
            \includesag{associativity-gc}
        \end{center}
    }
    On the level of ``forward'' maps, this commutative diagram encodes that morphism composition must be associative.
    One the level of ``backward'' maps, it implies that the following diagram must commute:
    \begin{center}
        \includesag{associativity-gc-back}
    \end{center}
    Let us look at the two different routes through this diagram.
    For the left-hand route, note that
    \begin{equation}
        \begin{aligned}
            \gcmBdec{(\enmof{}{\Obja}{\Objb}{\Objc}\gcPlus \catid_{\Objc\Objd})} & =\tupp{\gcmF{\catid_{\Objd}}\functhen(-)\functhen \gcmBdec{\enmof{}{\Obja}{\Objb}{\Objc}},\gcmFdec{\enmof{}{\Obja}{\Objb}{\Objc}}\functhen (-)\functhen \gcmB{\catid_{\Objc\Objd}}} \\
                                                                                 & =\tupp{(-)\functhen \gcmBdec{\enmof{}{\Obja}{\Objb}{\Objc}}, \gcmFdec{\enmof{}{\Obja}{\Objb}{\Objc}}\functhen (-)},
        \end{aligned}
    \end{equation}
    and so
    \begin{equation}\label{eq:left-hand-route-assoc}
        \begin{aligned}
            \gcmBdec{\enmof{}{\Objb}{\Objc}{\Objd}}\functhen \gcmBdec{(\enmof{}{\Obja}{\Objb}{\Objc}\gcPlus \catid_{\Objc\Objd})} & \colon
            \gcQ_{\Obja\Objd}\to \gcQ_{\Obja\Objc}^{\gcA_{\Objc\Objd}}\times \gcQ_{\Objc\Objd}^{\gcA_{\Obja\Objc}} \to \left(\gcQ_{\Obja\Objb}^{\gcA_{\Objb\Objc}}\times \gcQ_{\Objb\Objc}^{\gcA_{\Obja\Objb}}\right)^{\gcA_{\Objc\Objd}}\times \gcQ_{\Objc\Objd}^{\gcA_{\Obja\Objb}\times \gcA_{\Objb\Objc}} \\
            \gcq                                                                                                                  & \mapsto \tupp{\northena{\gcq}{(-)}, \northenb{(-)}{\gcq}}\mapsto \tupp{\northena{\gcq}{(-)}\functhen \gcmBdec{\enmof{}{\Obja}{\Objb}{\Objc}}, \gcmFdec{\enmof{}{\Obja}{\Objb}{\Objc}}\functhen \northenb{(-)}{\gcq}}.
        \end{aligned}
    \end{equation}
    For the right-hand route, note that
    \begin{equation}
        \begin{aligned}
            \gcmBdec{(\catid_{\Obja\Objb}\gcPlus \enmof{}{\Objb}{\Objc}{\Objd})} & =\tupp{\gcmFdec{\enmof{}{\Objb}{\Objc}{\Objd}}\functhen (-)\functhen \gcmB{\catid_{\Obja\Objb}}, \gcmF{\catid_{\Obja\Objb}}\functhen (-)\functhen \gcmBdec{\enmof{}{\Objb}{\Objc}{\Objd}}} \\
                                                                                 & =\tupp{\gcmFdec{\enmof{}{\Objb}{\Objc}{\Objd}}\functhen (-), (-)\functhen \gcmBdec{\enmof{}{\Objb}{\Objc}{\Objd}}},
        \end{aligned}
    \end{equation}
    and so
    \begin{equation}\label{eq:right-hand-route-assoc}
        \begin{aligned}
            \gcmBdec{\enmof{}{\Obja}{\Objb}{\Objd}}\functhen \gcmBdec{(\catid_{\Obja\Objb}\gcPlus \enmof{}{\Objb}{\Objc}{\Objd})} & \colon
            \gcQ_{\Obja\Objd}\to \gcQ_{\Obja\Objb}^{\gcA_{\Objb\Objd}}\times \gcQ_{\Objb\Objd}^{\gcA_{\Obja\Objb}}
            \to \gcQ_{\Obja\Objb}^{\gcA_{\Objb\Objc}\times \gcA_{\Objc\Objd}}\times \left( \gcQ_{\Objb\Objc}^{\gcA_{\Objc\Objd}}\times \gcQ_{\Objc\Objd}^{\gcA_{\Objb\Objc}}\right)^{\gcA_{\Obja\Objb}} \\
            \gcq                                                                                                                  & \mapsto \tupp{\northena{\gcq}{(-)},\northenb{(-)}{\gcq}}\mapsto \tupp{\gcmFdec{\enmof{}{\Objb}{\Objc}{\Objd}}\functhen \northena{\gcq}{(-)}, \northenb{(-)}{\gcq}\functhen \gcmBdec{\enmof{}{\Objb}{\Objc}{\Objd}} }.
        \end{aligned}
    \end{equation}

    Instead of now applying $\gcmB{\associator}$ directly, which is an obvious map but messy to write down, we evaluate the functions we obtained from our calculations for the left- and right-hand routes.
    Given $\tupp{\mora,\morb,\morc}\in \gcA_{\Obja\Objb}\times \gcA_{\Objb\Objc}\times \gcA_{\Objc\Objd}$, evaluating the two components of (\ref{eq:left-hand-route-assoc}) we find
    \begin{equation}
        (\gcq \northenasymb (-) \functhen \gcmBdec{\enmof{}{\Obja}{\Objb}{\Objc}})(\morc)=
        \gcmBdec{\enmof{}{\Obja}{\Objb}{\Objc}}(\northena{\gcq}{\morc})\colon \tupp{\morb,\mora}\mapsto \tupp{\northena{(\northena{\gcq}{\morc})}{\morb}, \northenb{\mora}{(\northena{\gcq}{\morc})}},
    \end{equation}
    and
    \begin{equation}
        (\gcmFdec{\enmof{}{\Obja}{\Objb}{\Objc}}\functhen \northenb{(-)}{\gcq})(\tupp{\mora,\morb})=\northenb{(\mora\functhen \morb)}{\gcq},
    \end{equation}
    respectively.

    For the right-hand route, evaluating (\ref{eq:right-hand-route-assoc}) gives
    \begin{equation}
        (\gcmFdec{\enmof{}{\Objb}{\Objc}{\Objd}}\functhen \northena{\gcq}{(-)})(\tupp{\morb,\morc})=\northena{\gcq}{(\morb\functhen \morc)},
    \end{equation}
    and
    \begin{equation}
        (\northenb{(-)}{\gcq} \functhen \gcmBdec{\enmof{}{\Objb}{\Objc}{\Objd}})(\mora)=\gcmBdec{\enmof{}{\Objb}{\Objc}{\Objd}}(\northenb{\mora}{\gcq}) \colon \tupp{\morc,\morb}\mapsto \tupp{\northena{(\northenb{\morb}{\gcq})}{\mora}, \northenb{\morb}{(\northenb{\mora}{\gcq})}}.
    \end{equation}
    By comparing the two routes, we obtain the conditions
    \begin{equation}
        \begin{aligned}
            \northenb{(\mora\functhen \morb)}{\gcq}    & = \northenb{\morb}{(\mora \northenbsymb \gcq)}, \\
            \northena{\gcq}{(\morb\functhen \morc)}    & =\northena{(\northena{\gcq}{\morc})}{\morb}, \\
            \northenb{\mora}{(\northena{\gcq}{\morc})} & =\northena{(\northenb{\mora}{\gcq})}{\morc}.
        \end{aligned}
    \end{equation}

    \paragraph*{Conditions from the unitality diagrams for enriched categories}
    Consider the right-hand portion of the diagram (\ref{eq:enrichment-unital2}), now for the case of \prcat.
    \proceedings{}{
        \begin{center}
            \includesag{unitality-gc}
        \end{center}
    }
    On the level of forward maps, this diagram encodes the condition that $\catid_\Obja\functhen \gca=\gca$ for any morphism $\gca \colon \Obja \mto \Objb$.
    On the level of backward maps, it amounts to the commutative diagram
    \begin{center}
        \includesag{lunitor-pn-flat}
    \end{center}
   Computing the right-hand route, we have
    \begin{equation}
        \begin{aligned}
            \gcmBdec{(\enidof{}{\Obja} \gcPlus \catid_{\homobject{}{\Obja}{\Objb}})}\colon \gcQ_{\Obja\Obja}^{\gcA_{\Obja\Objb}}\times \gcQ_{\Obja\Objb}^{\gcA_{\Obja\Obja}} & \to \singletonQ^{\gcA_{\Obja\Objb}}\times \gcQ_{\Obja\Objb}^{\singletonA} \\
            \tupp{\varphi_{\Obja\Obja},\varphi_{\Obja\Objb}}                                                                                                                 & \mapsto \tupp{!, \gcmF{\enidof{}{\Obja}}\functhen \varphi_{\Obja\Objb}\functhen \gcmB{\catid_{\homobject{}{\Obja}{\Objb}}}}=\tupp{!, \theonlyA \mapsto \varphi_{\Obja\Obja}(\catid_\Obja)},
        \end{aligned}
    \end{equation}
    and
    \begin{equation}
        \begin{aligned}
            \gcmBdec{((\enidof{}{\Obja} \gcPlus \catid_{\homobject{}{\Obja}{\Objb}})\functhen \enmof{}{\Obja}{\Obja}{\Objb})}\colon
            \gcQ_{\Obja\Objb} & \to \gcQ_{\Obja\Obja}^{\gcA_{\Obja\Objb}}\times\gcQ_{\Obja\Objb}^{\gcA_{\Obja\Obja}}\to \singletonQ^{\gcA_{\Obja\Objb}}\times \gcQ_{\Obja\Objb}^{\singletonA} \\
            \gcq              & \mapsto \gcmF{\enmof{}{\Obja}{\Obja}{\Objb}}(\gcq)=\tupp{\northena{\gcq}{(-)},\northenb{(-)}{\gcq}} \\
                              & \mapsto \tupp{\gcmF{\catid_{\homobject{}{\Obja}{\Objb}}}\functhen \northena{\gcq}{(-)} \functhen \gcmB{\enidof{}{\Obja}}, \gcmF{\enidof{}{\Obja}}\functhen \northenb{(-)}{\gcq}\functhen \gcmB{\catid_{\homobject{}{\Obja}{\Objb}}}} \\
                              & =\tupp{!,\theonlyA \mapsto (\northenb{\catid_\Obja}{\gcq})},
        \end{aligned}
    \end{equation}
    which, when compared with $\gcmB{\leftunitor} \colon \gcq \mapsto \tupp{!, \theonlyA \mapsto \gcq} $, gives the condition
    \begin{equation}
         \northenb{\catid}{\gcq}=\gcq.
    \end{equation}

    The left-hand portion of the diagram (\ref{eq:enrichment-unital2}) may be treated analogously and gives rise the to the conditions
    \begin{equation}
        \gca\functhen \catid_\Objb=\gca
        \qquad \text{and} \qquad
        \northena{\gcq}{ \catid}=\gcq.
    \end{equation}
\end{proof}

\begin{remark}
Exact nategories are those in which the implications in (\ref{eq:cond1}) and (\ref{eq:cond2}) are in fact equivalences. In the above proof, this corresponds to the morphisms (\ref{eq:proc3-maps}) and (\ref{eq:proc4-maps}) in \Bool being identities. 
\end{remark}

(\ref{prop:PN-enrich}) begs the question as to why we don't include the additional properties stated there as part of our definition of what a nategory is. Our reason is that these properties fail for examples of interest to us. 

For example, in applications it is normal that physical measurements and numerical representations on a computer are given only to a certain accuracy. This motivates (\ref{ex:berg-integer-version}), where we only allow integer values for norphisms. However, in that example, the properties (\ref{eq:compat-comp2}) and (\ref{eq:compat-comp1}) are not satisfied. To see this, consider (\ref{eq:compat-comp2})and consider morphisms $\morb$ and $\morc$ in $\Berg$ with $\lengthfun(\morb)=\lengthfun(\morc)=1.5$. For a norphism $\gcq$ of the appropriate signature we have 
\begin{equation}
\northena{\gcq}{(\morb\functhen \morc)}=\thefloor{\gcq - \lengthfun(\morb \mthen \mora)}=\thefloor{\gcq-3}
\end{equation}
on the one hand, and 
\begin{equation}
\northena{(\northena{\gcq}{\morc})}{\morb}=\thefloor{\thefloor{\gcq-1.5}-1.5}
\end{equation}
on the other. If we choose $\gcq=10$, for instance, the previous expressions evaluate to 7 and 6, respectively. 

As a different example, the properties (\ref{eq:neutrality1}) and (\ref{eq:neutrality2}) fail in Example \ref{exa:weak}. Indeed, there $\northenb{\catid_\Obja}{\nora}=\emptyset$ and $\northena{\nora}{\catid_\Objb}=\emptyset$, even though, in general, we would have $\nora \neq \emptyset$.

    \ifthenelse{\boolean{proceedings}}{}{
        \clearpage}
    \section{Conclusions}
This work showed that we can encode negative information in a categorical manner such that  norphisms (negative arrows) and morphisms (positive arrows) are equal citizens in the theory. Norphisms and morphisms are, in general, not mutually exclusive; they give complementary information.

We have seen how, in the category \Berg, norphisms can represent negative results such as lower bounds on distances between two locations.
A path planning algorithm must construct a morphism (a path) \emph{and} construct a norphism (a bound) to prove that the path is optimal.
We have also seen how, in the category \DP, norphisms can represent design impossibility results.

After defining nategories as categories with extra structure, we showed a way to encode this new concept using categories enriched in the dialecta category $\GSet$. This approach, however, introduces some compatibility properties for norphism composition that we do not wish to include in our general notion of nategory. Future work includes exploring if nategories might be recovered, on the nose, via a different enrichment, as well as studying various typical categorical concepts in the context of nategories. We would also be very happy to discover further interesting examples.

%
%
%
%

    \section*{Acknowledgments}
    The authors would like to thank David I.
    Spivak for fruitful discussions.
    We also thank Valeria de Paiva, Brendan Fong, and David Yetter for helpful comments.
    Gioele Zardini was supported by the Swiss National Science Foundation, under NCCR Automation, grant agreement 51NF40\_180545.

    \bibliographystyle{eptcs}
    \bibliography{generic}

\begin{thebibliography}{1}
\providecommand{\bibitemdeclare}[2]{}
\providecommand{\surnamestart}{}
\providecommand{\surnameend}{}
\providecommand{\urlprefix}{Available at }
\providecommand{\url}[1]{\texttt{#1}}
\providecommand{\href}[2]{\texttt{#2}}
\providecommand{\urlalt}[2]{\href{#1}{#2}}
\providecommand{\doi}[1]{doi:\urlalt{http://dx.doi.org/#1}{#1}}
\providecommand{\eprint}[1]{arXiv:\urlalt{https://arxiv.org/abs/#1}{#1}}
\providecommand{\bibinfo}[2]{#2}

\bibitemdeclare{article}{blass1995questions}
\bibitem{blass1995questions}
\bibinfo{author}{Andreas \surnamestart Blass\surnameend}
  (\bibinfo{year}{1995}): \emph{\bibinfo{title}{Questions and answers---a
  category arising in linear logic, complexity theory, and set theory}}.
\newblock {\sl \bibinfo{journal}{Advances in linear logic}}
  \bibinfo{volume}{222}, pp. \bibinfo{pages}{61--81},
  \doi{10.48550/ARXIV.MATH/9309208}.

\bibitemdeclare{book}{censi2022}
\bibitem{censi2022}
\bibinfo{author}{Andrea \surnamestart Censi\surnameend},
  \bibinfo{author}{Jonathan \surnamestart Lorand\surnameend} \&
  \bibinfo{author}{Gioele \surnamestart Zardini\surnameend}
  (\bibinfo{year}{2022}): \emph{\bibinfo{title}{Applied Compositional Thinking
  for Engineers}}.
\newblock \urlprefix\url{https://bit.ly/3qQNrdR}.
\newblock \bibinfo{note}{Work in progress book}.

\bibitemdeclare{article}{de1989dialectica}
\bibitem{de1989dialectica}
\bibinfo{author}{Valeria \surnamestart De~Paiva\surnameend}
  (\bibinfo{year}{1989}): \emph{\bibinfo{title}{The dialectica categories}}.
\newblock {\sl \bibinfo{journal}{Categories in Computer Science and Logic}}
  \bibinfo{volume}{92}, pp. \bibinfo{pages}{47--62}, \doi{10.48456/tr-213}.

\bibitemdeclare{inproceedings}{de1989dialecticabis}
\bibitem{de1989dialecticabis}
\bibinfo{author}{Valeria \surnamestart De~Paiva\surnameend}
  (\bibinfo{year}{1989}): \emph{\bibinfo{title}{A dialectica-like model of
  linear logic}}.
\newblock In: {\sl \bibinfo{booktitle}{Category Theory and Computer Science}},
  \bibinfo{organization}{Springer}, pp. \bibinfo{pages}{341--356}, \doi{10.1007/BFb0018360}.

\bibitemdeclare{incollection}{delling2009engineering}
\bibitem{delling2009engineering}
\bibinfo{author}{Daniel \surnamestart Delling\surnameend},
  \bibinfo{author}{Peter \surnamestart Sanders\surnameend},
  \bibinfo{author}{Dominik \surnamestart Schultes\surnameend} \&
  \bibinfo{author}{Dorothea \surnamestart Wagner\surnameend}
  (\bibinfo{year}{2009}): \emph{\bibinfo{title}{Engineering route planning
  algorithms}}.
\newblock In: {\sl \bibinfo{booktitle}{Algorithmics of large and complex
  networks}}, \bibinfo{publisher}{Springer}, pp. \bibinfo{pages}{117--139},
  \doi{10.1007/978-3-642-02094-0_7}.

\bibitemdeclare{article}{fong2018seven}
\bibitem{fong2018seven}
\bibinfo{author}{Brendan \surnamestart Fong\surnameend} \&
  \bibinfo{author}{David~I \surnamestart Spivak\surnameend}
  (\bibinfo{year}{2018}): \emph{\bibinfo{title}{Seven sketches in
  compositionality: An invitation to applied category theory}}.
\newblock {\sl \bibinfo{journal}{arXiv preprint arXiv:1803.05316}},
  \doi{10.48550/ARXIV.1803.05316}.

\bibitemdeclare{book}{kelly1982basic}
\bibitem{kelly1982basic}
\bibinfo{author}{Gregory~Maxwell \surnamestart Kelly\surnameend}
  (\bibinfo{year}{1982}): \emph{\bibinfo{title}{Basic concepts of enriched
  category theory}}.
\newblock \bibinfo{volume}{64}, \bibinfo{publisher}{CUP Archive}.

\bibitemdeclare{article}{shulman2018affine}
\bibitem{shulman2018affine}
\bibinfo{author}{Michael \surnamestart Shulman\surnameend}
  (\bibinfo{year}{2022}): \emph{\bibinfo{title}{Affine logic for constructive
  mathematics}}.
\newblock {\sl \bibinfo{journal}{The Bulletin of Symbolic Logic}}, pp.
  \bibinfo{pages}{1--51}, \doi{10.1017/bsl.2022.28}.

\end{thebibliography}

    \appendix
    \ifthenelse{\boolean{proceedings}}{}{
        \appendix
        \input{AA-proofs}}

\end{document}